\documentclass{article}
\usepackage{amsmath,amsfonts,amsthm,amssymb,amscd,bm,bbm, mathrsfs}
\usepackage[hidelinks]{hyperref}
\usepackage[usenames,dvipsnames]{color}
\usepackage{url}
 \hypersetup{colorlinks, citecolor=ForestGreen, linkcolor=MidnightBlue, urlcolor=Blue}

\DeclareMathAlphabet\mathbfcal{OMS}{cmsy}{b}{n}

\newcommand{\aA}{{\cal A}}
\newcommand{\KK}{{\cal K}}
\newcommand{\MM}{{\cal M}}
\newcommand{\LL}{{\cal L}}

\newcommand{\ty}{\infty}

\newcommand{\lag}{\langle}
\newcommand{\rag}{\rangle}
\newcommand{\la}{\lambda}
\newcommand{\DD}{{\cal D}}
\newcommand{\Z}{{\mathbb Z}}

\newcommand{\FF}{{\cal F}}
\newcommand{\IP}{{\mathbb P}}
\newcommand{\mek}{{\mathbf 1}}
\newcommand{\BBBBB}{{\mathcal B}}

\newcommand{\be}{\begin{equation}}
\newcommand{\ee}{\end{equation}}
\newcommand{\bp}{\begin{proof}}
\newcommand{\ep}{\end{proof}}
\newcommand{\bi}{\begin{itemize}}
\newcommand{\ei}{\end{itemize}}
\newcommand{\om}{\omega}

\newcommand{\mM}{\Lambda}
\newcommand{\pP}{\mathbb{P}}

\newcommand{\aaa}{{\cal A}}

\newcommand{\E}{\mathbb{E}}
\newcommand{\ccc}{{\cal C}}

\newcommand{\iii}{{\cal I}}

\newcommand{\kkk}{{\cal K}}

\newcommand{\mmm}{{\cal M}}
\newcommand{\nnn}{{\cal N}}

\newcommand{\ppp}{{\cal P}}
\newcommand{\PPPP}{\mathfrak P}

\newcommand{\C}{{\mathbb C}}

\newcommand{\e}{\mathbb{E}}

\newcommand{\pp}{\mathbb{P}}

\newcommand{\rr}{\mathbb{R}}

\newcommand{\PP}{{\cal P}}

\newcommand{\PPP}{\mathbf{P}}

\newcommand{\q}{\quad}

\newcommand{\f}{\frac}
\newcommand{\lm}{\lambda}

\newcommand{\p}{\partial}

\newcommand{\De}{\delta}

\newcommand{\diver}{\mathop{\rm div}\nolimits}

\newcommand{\supp}{\mathop{\rm supp}\nolimits}
\newcommand{\diam}{\mathop{\rm diam}\nolimits}

\newcommand{\Osc}{\mathop{\rm Osc}\nolimits}

\newcommand{\es}{\varepsilon}

\newcommand{\al}{\alpha}

\newcommand{\ch} {\mathbbm{1}}

\newcommand{\iin} {\infty}

\newcommand{\ef}{\eqref}
\newcommand{\dd}{\,{\textup d}}

\newcommand{\SSS}{{\mathscr S}}

\theoremstyle{plain}
\newtheorem{theorem}{Theorem}[section]

\newtheorem{lemma}[theorem]{Lemma}
\newtheorem{proposition}[theorem]{Proposition}

\theoremstyle{definition}

\theoremstyle{remark}
\newtheorem{remark}[theorem]{Remark}

\numberwithin{equation}{section}
\newtheorem*{definition*}{Definition}
\newtheorem*{problem*}{Problem}

\newtheorem*{remark*}{Remark}
\newtheorem*{note*}{Note}
\begin{document}
\author{D.~Martirosyan\footnote{Inria of Paris, 2 rue Simone Iff, 75012 Paris, France; e-mail: \href{mailto:martirosyan.davit@gmail.com}{Martirosyan.Davit@gmail.com}} \and V.~Nersesyan\footnote{Laboratoire de Math\'ematiques, UMR CNRS 8100, UVSQ, Universit\'e Paris-Saclay,  45  Av.  des Etats-Unis, 78035, Versailles, France;  e-mail: \href{mailto:Vahagn.Nersesyan@math.uvsq.fr}{Vahagn.Nersesyan@math.uvsq.fr}} 
}

\title{Multiplicative ergodic theorem for a non-irreducible random dynamical system}

\date{\today}

 \maketitle

\begin{abstract}

We study the asymptotic properties of the trajectories of a  discrete-time   random dynamical system  in an infinite-dimensional Hilbert space.  Under some natural assumptions on the model, we establish a  multiplicative ergodic theorem with an exponential rate of convergence.~The   assumptions are satisfied for a large class of parabolic PDEs, including the 2D  Navier--Stokes and  complex Ginzburg--Landau equations perturbed by a non-degenerate bounded random kick force.~As a consequence of this ergodic theorem,  we derive   some new results  on  the  statistical properties of the trajectories  of the underlying  random dynamical system.
In particular, we   obtain  large deviations principle  for the occupation measures and   the analyticity of the pressure function in a setting where the system is not irreducible.~The proof relies on a refined version of the uniform Feller  property combined with some  contraction and   bootstrap arguments.


\end{abstract}


\setcounter{section}{-1}

 \section{Introduction} 
This paper continues the study, initiated in \cite{JNPS-2012}, of the   multiplicative ergodicity   for
     the  following discrete-time random dynamical system 
\begin{equation} \label{E:0.1}
u_k=S(u_{k-1})+\eta_k, \quad k\ge1. 
\end{equation}Here   $\{\eta_k\}$ is a sequence of bounded  independent identically distributed (i.i.d)  random variables in a   separable Hilbert space~$H$ and
$S:H\to H$ is a continuous mapping subject to some natural assumptions.~Without going into formalities here, let us only mention that these assumptions ensure that   $S$ is locally Lipschitz possessing some dissipative and regularising properties,  and that the   random perturbation $\{\eta_k\}$ is non-degenerate  (see   Conditions (A)-(D) in the next section).
 A large  class of  dissipative PDEs with some discrete random  perturbation can be written in the form \ef{E:0.1}. This class includes 
     the 2D  Navier--Stokes  and  the complex Ginzburg--Landau   equations   (see Section \ref{S:5.5}).
 
 \smallskip
 System \eqref{E:0.1} defines   a homogeneous family of Markov  chains~$(u_k,\pP_{u})$ para\-met\-rised by the initial condition.~It is well known that this family 
  admits~a~unique stationary measure~$\mu$ which is {\it exponentially mixing} in the sense that~for   any $1$-Lipschitz function $f:H\to \rr$  and any initial condition $u_0=u \in H$, we~have  
 \begin{equation}\label{E:0.2}
  \left|\E_u f(u_k)- \int_H f(v)\mu(\!\dd v) \right|\le C (1+\|u\|) e^{-\alpha k} \quad k\ge0,
 \end{equation} where $C$ and $\alpha$ are some   positive numbers. 
We   refer the reader to the papers~\cite{FM-1995,KS-cmp2000,EMS-2001,BKL-2002} for the first results of this type    and  Chapter~3 of the book~\cite{KS-book}    for    details on the problem of ergodicity for~\eqref{E:0.1}.

  \smallskip
  In this paper,    motivated by applications to  {\it large deviations},  we consider the asymptotic behavior of the product 
\begin{equation}\label{E:0.3}
\Xi_k^Vf =f(u_k)\exp\left(\sum_{n=1}^kV(u_n)\right),
  \end{equation}
 where $V:H\to\rr$ is a  given bounded Lipschitz-continuous function (potential). We   say that~$(u_k,\pP_{u})$  satisfies a {\it multiplicative ergodic theorem} (MET)    if we have a limit     of the form 
 \begin{equation}\label{E:0.4}
\la_V^{-k} \,\PPPP_k^Vf(u)\to h_V(u)\int_H f(v) \mu_V(\!\dd v) \quad\text{ as $k\to \infty$}
  \end{equation}  for some number $\lambda_V>0$,   positive continuous function $h_V:H\to \rr $, and a Borel probability measure $\mu_V$   not depending on $f$ and $u$, where $\PPPP_k^V$ is the {\it Feynman--Kac semigroup} associated with potential $V$, that is $\PPPP_k^Vf(u)=\e_u\left\{ \Xi_k^Vf\right\}$.   This is a form of Oseledec's theorem obtained in \cite{O68}  (see also \cite{Ruelle, Mane, Kening}).

\smallskip
 Let us recall that, under Conditions~(A)-(D),  a MET   is established in 
\cite{JNPS-2012}, in the case when the initial point~$u$ belongs to the support $\aA$  of the stationary measure~$\mu$. The set $\aA$   is  compact     in~$H$ and it is  the smallest invariant set for  system~\eqref{E:0.1}. We extend this result in three directions.~Our first main result shows that under  the same  hypotheses,   convergence 
\ef{E:0.4}    is exponential for $u\in \aaa$.~We next show that if the oscillation of the potential $V$ is sufficiently small, then the convergence  holds true for any~$u\in H$. Moreover, the rate of convergence is again exponential. Finally, our third   result proves that this restriction on the oscillation of $V$ can be dropped if the operator $S$  is a  subcontraction\,\footnote{ See Condition~(E) in the next section for the definition and Section \ref{S:5.5} for examples.}
  with respect to a certain    metric whose topology is weaker than the natural one on~$H$.

\smallskip 

We provide a section with different applications.   Probably the most important one  is the large deviations principle  (LDP).~Our multiplicative ergodic theorems allow   to prove the existence and   analyticity of the pressure function.~Combining this  with  the  G\"artner--Ellis theorem, the Kifer's criterion (see~\cite{kifer-1990}),      and some techniques developed in~\cite{JNPS-2012}, we obtain~new local and global large deviations results for the occupation measures    in a  {\it non-irreducible}\,\footnote{Recall that~$(u_k,\pP_{u})$ is irreducible, if for any $u\in H$ and any ball $B\subset H$, we have $\pP_u\{u_l\in B\}>0$ for some $l\ge1$.} setting.     Previously, the   LDP 
  for randomly forced PDEs has been studied   in the papers~\cite{gourcy-2007a,gourcy-2007b,MN-2015}, in the continuous-time, and in~\cite{JNPS-2012,JNPS-2013,JNPS-2014}, for the  discrete-time case. Let us also mention   the papers~\cite{wu-2001, MT-2005},  where~the
  LDP  is derived from the MET 
   for    Markov processes possessing  the strong~Feller
    property. 
    In all these references the underlying Markov processes are~{\it irreducible}.
   
\smallskip
  We also give two other applications~related to the   random time in the~strong law of large~numbers and the speed of attraction of the support of the stationary~measure.

\smallskip
  Let us     outline some ideas behind the proof of the multiplicative~ergodicity.~From the boundedness of the random variables $\eta_k$ it follows that system~\eqref{E:0.1} has a    compact invariant absorbing set $X$. This  allows to reduce our study to a compact phase space   $X$.~The function~$h_V$ and the measure~$\mu_V$  are  the eigenvectors of the Feynman--Kac semigroup   corresponding to an   eigenvalue~$\lm_V$,~i.e.,  
  $$
\PPPP_1^V h_V=\lambda_V h_V, \,\,  \PPPP_1^*{\mu_V}=\lambda_V\mu_V,
$$
where $\PPPP_1^{V*}$ is the dual of $\PPPP_1^V$.~By normalising   $\PPPP_k^V$, we reduce    the 
problem~to
the study of the exponential mixing for the following auxiliary Markov semigroup: 
$$
\SSS_k^V g=  \lm_V^{-k} h_V^{-1}\PPPP^V_k(gh_V), \q\q\  g\in C(X).
$$ The latter is   achieved    by showing that, for sufficiently large time $k$,  the dual operator~$\SSS_k^{V*}$ is a contraction in the space of Borel probability measures endowed with the dual-Lipschitz norm. The proof of the contraction relies on    the following four   ingredients   (see Theorem \ref{T:2.1}): 
 (i)	{\it refined uniform Feller property}, (ii)	{\it uniform irreducibility on $\aA$},
    (iii)	{\it concentration near~$\aA$},
(iv)	  {\it exponential bound}.
  The first one  is an enhanced    version  of the uniform Feller property from~\cite{KS-cmp2000, JNPS-2012} with specified constants. Its proof relies on the coupling method.    
     In the particular case~$V=0$, it coincides with the {\it asymptotic strong Feller property} in  \cite{HM-2006,HM-2008}. Property (ii)    follows from the dissipativity of the system and is   well known. The lack of irreducibility on the set~$X\setminus \aA$ is compensated by       (iii) and (iv).   Their verification is highly non-trivial, especially when there is no restriction on the oscillation of the potential~$V$. We use   the subcontraction condition on $S$   and a bootstrap argument to derive them for~\eqref{E:0.1}.

\smallskip
 Let us note that  the proof of the existence of an eigenvector $\mu_V$ is standard, it follows by a simple application of the Leray--Schauder fixed point theorem.  On the other hand, the existence of~$h_V$    is more delicate. It is derived from the above-mentioned properties (i), (ii), and (iv).

  Finally,  let us mention that the long-time behavior of the Feynman--Kac semigroup has been studied in the literature in the case when the   Markov process is  strong Feller and the space of probability measures is endowed with the  total variation metric (e.g., see \cite{wu-2001, M-04}). Obviously,  these results cannot be applied in our setting.

\smallskip
The paper is organised as follows.~In  Section~\ref{S:1}, we formulate our main results on the  multiplicative ergodicity.~In Section \ref{S:2}, we establish  an abstract exponential convergence criterion   for generalised Markov semigroups, which is applied in   Section~\ref{S:3} to prove the main      results.~Section \ref{S:4} is devoted to the proof of the refined uniform Feller property.  In Section~\ref{S:5}, we present the above-mentioned  applications of the   MET   and discuss some examples of PDEs verifying Conditions (A)-(E).

      \subsection*{Acknowledgments} 
      
     The authors   thank  Vojkan Jak\v si\'c, Philippe Robert, and Armen Shirikyan    for useful remarks. The research of DM was partially supported by the Celtic Plus project SENDATE TANDEM (C2015/3-2). The research of VN   was supported by the ANR grant NONSTOPS ANR-17-CE40-0006-02 and CNRS PICS grant Fluctuation theorems in stochastic systems.

\subsection*{Notation}
 For a Polish space $(H,d)$, we shall use the following standard notation.

\smallskip
\noindent
$C(H)$ is the space of continuous functions $f:H\to \rr$.   For any    $ X\subset H$ and $f\in C(H)$, we denote $\|f\|_{X}=\sup_{u\in   X}|f(u)|$ and write $\|f\|_{\ty}$ instead~of~$\|f\|_{H}$.
 
\smallskip
\noindent 
$C_b(H)$  is the space of bounded  functions  $f\in C(H)$  with the norm $\|f\|_\infty$.
 
\smallskip
\noindent
$L_b(H)$   is the space of functions $f\in C_b(H)$ for which the following
norm is finite 
$$
 \|f\|_{L}=\|f\|_\infty+\sup_{u\neq v} \frac{
{|f(u)-f(v)|}}{ d(u,v)  }.
$$    $\MM_+(H)$ is the set   of non-negative finite Borel measures on $H$ endowed with the dual-Lipschitz metric  
  $$
  \|\mu_1-\mu_2\|_L^*=\sup_{\|f\|_L\le 1} |\lag f,\mu_1\rag-\lag f,\mu_2\rag|, \quad \mu_1, \mu_2\in \MM_+(H),
  $$  where  $\lag f,\mu\rag=\int_Hf(u)\mu(\!\dd u)$.
$\PP(H)$ is the subset of Borel probability measures. 

\smallskip
\noindent
  $\Osc(f)$ is the oscillation of a function $f:H\to \rr$, that is the number defined by~$\sup_{u\in H} f(u)-\inf_{u\in H} f(u)$.

\smallskip
\noindent
    $B_R(a)$   is the closed   ball in~$H$ of radius~$R$ centered at~$a$. When $H$ is  Banach   and $a=0$, we write $B_R$ instead of  $B_R(0)$.

 \section{Main  results}\label{S:1}

We consider problem \eqref{E:0.1} in a  separable Hilbert space $H$   endowed with the scalar product $(\cdot,\cdot)$  and the associated norm $\|\cdot\|$.~We   assume that     the following  hypotheses  are satisfied for the mapping~$S:H\to H$ and
the sequence~$\{\eta_k\}$; they are exactly the same as in~\cite{JNPS-2012}:

    \medskip
    \noindent
{\bf (A)} 
{\sl $S$ is locally Lipschitz in $H$,   and for any $R>r>0$, there is a positive number $a<1$ and an integer $n_0\ge1$ such that
\begin{gather}
\|S^n(u)\|\le a\|u\|\vee r\q\quad\mbox{for $u\in B_R$, $n\ge n_0$},
\label{E:1.2}
\end{gather}
where $S^n$ stands for  the $n^{\mathrm{th}}$ iteration of~$S$.}

\smallskip
For any     set $B\subset H$,  let us  define the sequence of sets
$$
\aA(0,B)=B, \quad \aA(k,B)=S(\aA(k-1,B))+\KK, \quad k\ge1, 
$$where~$\KK$ is the support of the law of~$\eta_1$.
We denote by~$\aA(B)$ the closure in~$H$ of the union of the sets~$\aA(k,B),  k\ge1$ and   call it   {\it the domain of attainability~from~$B$}.
  
\medskip
\noindent
{\bf (B)}
{\sl There is a number $\rho>0$ and a   
  continuous
 function~$k_0=k_0(R)$ such that}
\begin{equation} \label{E:1.3}
\aA(k,B_R)\subset B_\rho\quad\mbox{\sl for $R\ge0$,  $\,k\ge k_0(R)$}. 
\end{equation}

\noindent 
{\bf (C)}~{\sl There is an orthonormal basis~$\{e_j\}$ in~$H$ such that, for any $R>0$,
\begin{equation} \label{E:1.4}
\|(I-{\mathsf P}_N)(S(u_1)-S(u_2))\|\le \gamma_N(R)\|u_1-u_2\|, \q u_1, u_2\in B_R,
\end{equation}
where ${\mathsf P}_N$ is the orthogonal projection onto     $\textup{span}\{e_1,\dots,e_N\}$, and $\gamma_N(R)\downarrow 0$ as~$N\to\infty$.}

\smallskip
\noindent
{\bf (D)}~{\sl  The random variable~$\eta_k$ is of the form $\eta_k=\sum_{j=1}^\infty b_j\xi_{jk}e_j,$
where $\{e_j\}$ is the orthonormal basis from~{\rm(C)}, $b_j>0$ are numbers such that~$ \sum^\infty_{j=1} b_j^2<\infty,
$ and~$\xi_{jk}$ are independent scalar random variables
 with  law having  a density $p_j$ with respect to the Lebesgue measure.~We assume that $p_j$   is  a   continuously differentiable  function  such that $p_j (0) > 0$ and with support in $[-1,1]$.} 

\medskip

 The reader is referred to the beginning of Section \ref{S:3}   for some comments about these conditions and to Section~\ref{S:5.5} for examples.

 \smallskip     
      Let $\aA=\aA(\{0\})$  be the domain of attainability from zero.~This set is  compact  in $H$ and invariant for \eqref{E:0.1}, i.e.,   for any~$u_0\in \aA$, we have~$u_k\in \aA$~for all~$k\ge1$ almost surely. Recall that  $\PPPP_k^V:C_b(H)\to C_b(H)$ 
       is the       Feynman--Kac semigroup   associated with   $V\in L_b(H)$, that is,  
$$
  \PPPP_k^Vf(u)=\e_u\left\{ \Xi_k^Vf\right\}=\e_u\left\{ f(u_k)\exp\left(\sum_{n=1}^kV(u_n)\right)\right \}
$$
  and  $\PPPP_k^{V*}:\MM_+(H )\to \MM_+(H )$~stands for its dual.
      The following   theorem  is our first result.     
\begin{theorem} \label{T:1.1}
Under   Conditions {\rm(A)-(D)},  system \eqref{E:0.1} is multiplicatively ergodic on $\aA$ with any potential $V\in  L_b(\aA)$, i.e.,  the following two assertions~hold.
\begin{description}
\item[Existence.] There is a number $\lambda_V\!\!>\!0$, a   measure $\mu_V\!\in\!\PP(\aA)$ whose support~coin\-cides~with~$\aA$, and a   positive function $h_V\in L_b(\aA)$   such that for any~$u\in \aA$,
\begin{equation}
\PPPP_1^V h_V(u)=\lambda_V h_V(u), \quad  \PPPP_1^{V*}\mu_V=\lambda_V\mu_V. \label{E:1.5}\\
   \end{equation}\item[Exponential convergence.]There are  positive numbers  
    $\gamma_V$ and $C_V$ such~that  
 \be\label{E:1.6}
\left|\lm_V^{-k}\PPPP_k^V f(u)-\lag f, \mu_V\rag h_V(u)\right|\le C_V e^{-\gamma_V k}, \q\q k\ge 1
\ee for any    $u\in \aA$ and $f \in L_b(\aA)$ with $\|f\|_L\le 1$.
 \end{description}
 \end{theorem}
 The  second result establishes multiplicative ergodicity     on the   whole~spa\-ce~$H$, under the restriction that  the oscillation of the potential is~small.  
\begin{theorem} \label{T:1.2}
Under   Conditions {\rm(A)-(D)}, there is a number $\delta>0$ such that system \eqref{E:0.1} is multiplicatively ergodic on   $H$ with any potential $V\in  L_b(H)$ satisfying  $\Osc(V)\le\delta$. Namely,  we have the following.
\begin{description}
\item[Existence.] There is a number $\lambda_V>0$, a  measure $\mu_V\in\PP(H)$ whose support coincides with~$\aA$, and  a  positive function $h_V\in C(H)$  such that \eqref{E:1.5}
  holds  for any $u\in H$.  

\item[Exponential convergence.] 
There is    
    $\gamma_V>0$ such that for any $R>0$, we have inequality \eqref{E:1.6}   
  for any $u\in B_R$,  $f \in L_b(H)$ with $\|f\|_L\le 1$,  and some number~$C_{V}(R)>0$.
 \end{description}
 \end{theorem}

Our third result  shows that
   the restriction on the smallness of~$\Osc(V)$ in this  theorem     can be removed  if we  assume additionally that $S$ is a subcontraction with respect to some metric.   
   
 \medskip
{\bf (E)}~{\sl There is a  translation invariant\,\footnote{i.e., $d'(u+w,v+w)=d'(u,v)$ for any $u,v,w\in H$.}~metric~$d'$ on $H$ whose topology is   weaker  
   than the natural one on $H$   such that    
\be\label{E:1.7}
d' (S(u), S(v))\le  d'(u, v) \q\q \text{ for   } u, v\in  \aA(B_{\rho+1}), 
\ee where $\rho$ is the number in   {\rm(B)}.}

 \begin{theorem}\label{T:1.3}
Under Conditions  {\rm(A)-(E)}, system \eqref{E:0.1} is multiplicatively ergodic on   $H$ with any potential $V\in  L_b(H)$.
\end{theorem} These three
theorems are the main results of this paper.~Their proof is based on   an abstract exponential convergence result    presented in the next section.

 \section{Exponential convergence for generalised Markov semigroups}\label{S:2}

In Theorem 2.1 in \cite{JNPS-2012}, a convergence result  is established for generalised Markov semigroups in a  compact metric space $(X,d)$. In this section, we extend that result by proving exponential convergence, under some natural additional hypotheses.   Moreover, we   relax the  irreducibility condition which now holds only on some~subset.
 
\smallskip
Recall   that  a family  $\{P(u,\cdot),u\in X\}\subset\MM_+(X)$ is a {\it generalised Markov kernel\/}  if the function   $u\mapsto P(u,\cdot)$  from~$X$ to~$\MM_+(X)$ is continuous and non-vanishing.   To any such kernel we associate  semigroups 
$\PPPP_k: C(X)\to C(X)$ and~$\PPPP_k^*: \MM_+(X)\to \MM_+(X)$ by
$$
\PPPP_kf(u)=\int_XP_k(u,\!\dd v)f(v), \quad 
\PPPP_k^*\mu(\Gamma)=\int_XP_k(u,\Gamma)\mu(\!\dd u), 
$$
where $P_k(u,\cdot)$ are the iterations of $P(u,\cdot)$.   
   \begin{theorem} \label{T:2.1}
Let $\{P(u,\cdot),u\in X\}$ be a generalised Markov kernel and $\aA$ a closed set in $X$ such that $P(u,X\setminus \aA)=0$ for   $u\in \aA$.~Assume that the   following  conditions are satisfied.
\begin{description}
\item[(i) Refined uniform Feller property.]
For any   
   $c\in(0,1)$, there is~$C>0$  such that  for any $f\in L_b(X)$ and~$v, v'\in X$, we have
  \begin{equation}
|\PPPP_k f(v)-\PPPP_k f(v')|  \le \left( C\, \|f\|_X +c  \, \|f\|_L\right) \|\PPPP_k{\mek}\|_X \,d(v,v').\label{E:2.1} 
\end{equation}
\end{description}
\begin{description}
\item[(ii) Uniform irreducibility on $\aA$.] 
For any $r>0$, there is an integer $m\ge1$ and a number $p>0$ such that
$$
P_m(u,B_r(\hat u))\ge p\quad\mbox{for all $u\in X$ and $\hat u\in \aA$}. 
$$ 
\end{description}
Then there are positive numbers   $\lambda, \gamma, C$, a measure $\mu\in\PP(X)$ whose support coincides with~$\aA$, and a positive function $h\in L_b(\aA)$  such that  for any $u\in \aA$ and $\nu\in\ppp(\aA)$, we have
\begin{gather}
\PPPP_1 h(u)=\lambda h(u), \quad \PPPP_1^*\mu=\lambda\mu,\label{E:2.2}\\
\|\lambda^{-k}\PPPP_k^*\nu -\lag h,\nu\rag\mu\|_L^*\le C e^{-\gamma k}. \label{E:2.3}
\end{gather}Furthermore, 
 let us assume additionally the following conditions.
\begin{description}
 \item[(iii) Concentration near $\aA$.]The following limit holds
 \begin{equation}\label{E:2.4}
\lim_{k\to \iin}\|P_k(\cdot,  X\setminus\aA_r)\|_X=0	 \q\quad\text{for any $r>0$} ,
\end{equation} 
where      $\aaa_r$ is the $r$-neighborhood of $\aaa$: 
 \begin{equation}\label{E:2.5}
 \aA_r=\left\{u\in X: \inf_{v\in \aA} d(u,v)< r\right\}.
\end{equation}

\end{description}

\begin{description}
\item[(iv)   Exponential bound.] We have 
\begin{equation}\label{E:2.6}
\mM=\sup_{k\ge 1} \lambda^{-k}  \|\PPPP_k \mek\|_X  <\ty.
\end{equation}
\end{description}
Then $h$ has a positive Lipschitz-continuous extension to the   space $X$ (again denoted by $h$) such that \eqref{E:2.2} and \eqref{E:2.3} hold for any $u\in X$ and $\nu \in \PP(X)$. \end{theorem}
\begin{remark}\label{R:2.2}
Let us underline that in condition (ii), the initial point $u$ belongs to~$X$ and the final one~$\hat u$ to~$\aA$, so the semigroup is not   irreducible in~$X$.~Also note that~(iii) and (iv) are necessary conditions for convergence~\eqref{E:2.3}. Indeed,~\eqref{E:2.3} is equivalent~to 
$$
\sup_{\|f\|_L\le1 }\|\lambda^{-k}\PPPP_k f -\lag f,\mu\rag h\|_X\le C e^{-\gamma k}.
$$By taking any non-negative function $f\in L_b(X)$  that vanishes on   $\aA$ and  equals~$1$ outside   $ \aA_r$, we get (iii), and taking $f=\mek$, we get (iv).
 
 Furthermore,  in the case $X=\aA$, (iii) is trivially satisfied and (iv)    follows from (i) and (ii). Indeed, if   conditions  (i) and (ii) hold,   we can apply  Theorem~2.1 in~\cite{JNPS-2012} on the set $\aA$:
\begin{equation}\label{E:2.7}
 \|\lambda^{-k}\PPPP_k f  - \lag f,\mu\rag h\|_\aA\to 0 \quad \text{as $k\to \ty$ for   $f\in L_b(\aA)$}.
\end{equation}
 Choosing  $f=\mek$, we get (iv).
 \end{remark}
 
 \medskip
We conclude from this remark that it   suffices to prove only the second assertion of Theorem \ref{T:2.1}.~Its proof is divided into three~parts.

 \subsection{Existence of eigenvectors $\mu$ and $h$}
The existence of an eigenvector~$\mu\in \PP(\aA)$ is shown in~\cite{JNPS-2012}, by applying the Leray--Schauder theorem  to  the continuous mapping
 $$
 F:\PP(\aA)\to\PP(\aA), \quad F(\mu)=(\PPPP_1^*\mu(\aA))^{-1}\PPPP^*_1\mu.
 $$ Any  fixed point~$\mu$ of $F$ is an eigenvector for $\PPPP_1^*$ corresponding to an eigenvalue~$\lambda=\PPPP^*\mu(\aA)>0$. The irreducibility on $\aA$ implies that  $\supp\mu=\aA.$~Note that replacing $P(u,\Gamma)$ by $\lambda^{-1}P(u,\Gamma)$, we may assume that $\lm=1$.   From now on, we shall always assume    that $\lm=1$, without further stipulation.

 \smallskip
 
 Let us show the existence of a  Lipschitz-continuous    function~$h:X\to\rr_+$ satisfying~\eqref{E:2.2} for any $u\in X$. We    use      some arguments from~\cite{JNPS-2012}.  
 
  \smallskip
  {\it Step~1: Existence}.~Using   \ef{E:2.1}   and   \ef{E:2.6},  we see that   the sequence~$\{\PPPP_k\mek\}$ is uniformly equicontinuous on~$X$. It follows that so is the sequence 
\begin{equation}\label{E:2.8}
h_k=\frac1k\sum_{n=1}^{k}\PPPP_n\mek. 
\end{equation}
Applying the Arzel\`a--Ascoli theorem, we construct a function~$h: X \to\rr_+$ and a sequence $k_j\to\infty$ such that 
\begin{equation}  \label{E:2.9}
\|h_{k_j}-h\|_X\to0\quad\mbox{as $j\to\infty$}. 
\end{equation}
  Passing to the limit as $j\to \ty$ in the equality  
$$
\PPPP_1 h_{k_j}(u)  
=h_{k_j}(u)+\frac{1}{k_j}\bigl(\PPPP_{k_j+1}\mek(u)- \PPPP_1 \mek(u)\bigr), \quad u\in X, 
$$  
     and using      \eqref{E:2.6},    we get~\eqref{E:2.2}.
     The Lipschitz-continuity of
   $h$   follows from~\eqref{E:2.1}, \eqref{E:2.6}, \eqref{E:2.8},  and \eqref{E:2.9}.

\smallskip
{\it Step~2: Positivity}.  
Now we   show that  $h(u)>0$ for any $u\in X$.
Using equalities $\PPPP^*_k\mu=  \mu$ and \eqref{E:2.8}  together with    limit \eqref{E:2.9},     we obtain   
$$
1=\lag  h_{k_j},\mu\rag\to \lag h,\mu\rag\quad\text{as $j\to\ty$}.
$$ So $\lag h,\mu\rag=1$ and $h(\hat u)>0$ for some $\hat u\in \aA$. 
By the continuity of $h$, there~is~$r>0$ such that~$h(v)\ge r$ for any~$v\in B_r(\hat u)$. Thanks to~(ii), for sufficiently large~$m\ge1$,   
$$
h(u)=\PPPP_mh(u)
\ge r\,   P_m\bigl(u,B_{r}(\hat u)\bigr)>0 \quad\text{ for any $u\in X$}.
$$ 
   This completes the proof of existence of eigenvectors    $\mu$ and $h$.
      The uniqueness   will follow from inequality \eqref{E:2.3}.

 \subsection{A weak version of \ef{E:2.3}}
   In this section, we show that   the left-hand side of \ef{E:2.3} converges to zero.  We shall   use this in the next section    to establish exponential convergence.
 \begin{proposition}\label{P:2.3}
Under   Conditions {\rm (i)-(iv)}, we have 
  \be\label{E:2.10}
\sup_{\nu\in \PP(X)}\|\PPPP_k^*\nu -\lag h,\nu\rag\mu\|_L^*\to 0\quad\text{ as $k\to \infty$}.
\ee  
 \end{proposition}
 \bp  It suffices to prove the   limit
\be\label{E:2.11}
 \sup_{\|f\|_L\le 1, \, \lag f,\mu\rag=0}\|\PPPP_k f\|_X\to0
\q\quad\mbox{as~$k\to\infty$}.
\ee  Let us  represent $k=l+m$, use the semigroup property,   and write for any~$r>0$,
\begin{equation}\label{E:2.12}
 \PPPP_k f (u) =  \PPPP_l \left(\ch_{\aA_r}  \PPPP_m f \right)(u)+ \PPPP_l \left(\ch_{X\setminus \aA_r}  \PPPP_m f \right)(u)=\iii_1+\iii_2,  
\end{equation}
  where $\aaa_r$ is defined by \eqref{E:2.5}.

\smallskip
{\it Step~1:~Estimate~for~$\iii_1$}.~Inequalities \ef{E:2.1}, \eqref{E:2.6}, and $\|f\|_L\le 1$ imply that 
$$
 |\PPPP_m f(v)-\PPPP_m f(v')| \le C_1\mM\,d(v,v'), \quad m\ge1, v,v'\in X
$$
for some   number $C_1>0$.~Combining this with the definition of $\aA_r$, we see~that 
$$
|\iii_1|\le C_1 \mM \,r\,   \PPPP_l \mek(u)+ \|\PPPP_m f\|_\aA \, \PPPP_l \mek(u).
$$
Taking   the supremum over $u\in X$ and using \eqref{E:2.6}, we~get
$$
\|\iii_1\|_X\le  C_1 \mM^2 \,r\,    + \|\PPPP_m f\|_\aA \,\mM.  
$$
Moreover,  by virtue of \eqref{E:2.7},
$$
 \sup_{\|f\|_L\le 1, \, \lag f,\mu\rag=0}\| \PPPP_m f\|_\aA\to0
\q\quad\mbox{as~$m\to\infty$},
$$ 
whence
\begin{equation}\label{E:2.13}
 \sup_{\|f\|_L\le 1, \, \lag f,\mu\rag=0}\|\iii_1\|_X\le  C_2 \, r
\end{equation}
for arbitrary $r>0$ and $l\ge1$, and sufficiently large $m=m(r)\ge1$.

\smallskip
{\it Step~2:  Estimate for $\iii_2$}.  Let us fix $m$ such that \eqref{E:2.13} holds.~Using the~inequa\-lities 
 $\|f\|_X\le 1$ and  \eqref{E:2.6} together with \eqref{E:2.4}, we obtain
    $$
  \iii_2\le \Lambda\, \|P_l(\cdot, X\setminus \aA_r)\|_X\to 0\q\quad\mbox{as~$l\to\infty$}.
  $$ 
 Combining this with \ef{E:2.12} and \ef{E:2.13}, we get
$$
\|\PPPP_kf\|_X\le \|\iii_1\|_X+\|\iii_2\|_X\le 2\, C_2 \,r \,
$$
for $k=k(r)\ge 1$ sufficiently large. Since $r>0$ can be chosen arbitrarily small, we arrive at \ef{E:2.11}.

\ep

\subsection{The rate of  convergence} 

In this section, we show that the rate of convergence in \eqref{E:2.10} is exponential. To this end,
we introduce  an auxiliary   Markov semigroup~$\SSS_k$ acting on~$C(X)$~by  
$$
\SSS_k g= h^{-1}\PPPP_k(gh) \q\q\text{ for } g\in C(X).
$$
The following result shows that $\SSS_k$ is exponentially mixing.
\begin{proposition}\label{P:2.4}Under the conditions of Theorem \ref{T:2.1}, 
 the semigroup $\SSS_k$ has a unique stationary measure   given by $\hat \sigma =h \mu$. Moreover, we have
\be\label{E:2.14}
\| \SSS^*_k\sigma-  \hat\sigma\|_L^*\le Ce^{-\gamma k} \q\q\text{ for }\sigma  \in \ppp(X), \, k\ge 1,
\ee where $C>0$  and $\gamma>0$ are some numbers not depending on $\sigma $  and $k.$ 
\end{proposition}
Taking this proposition for granted, let us   prove \eqref{E:2.3}.     Choosing $\sigma=\delta_u$ in~\eqref{E:2.14}, where    $\delta_u$ is the Dirac measure
concentrated at $u\in X$,
  we see that
$$
 | \PPPP_k(gh)(u) -\lag gh,  \mu\rag h(u)   | \le Ce^{-\gamma k} h(u)\le C_1e^{-\gamma k}, \q\q\text{   $ k\ge 1$}
$$ for any   $g \in  L_b(X)$ with   $\|g\|_L\le 1$.~Since $h$ is positive and Lipschitz, any $f\in L_b(X)$ can be represented as $f=g h$ for some $g\in L_b(X)$, which leads to~\eqref{E:2.3}.

\medskip
The rest of the section is devoted to the proof of Proposition \ref{P:2.4}. Note that the equality
$$
 \SSS^*_1\sigma=h\PPPP_1^*(h^{-1}\sigma)
$$and the fact that $\mu$ is an eigenvector for $ \PPPP^*_1$
imply that $\hat \sigma =h \mu$ is a stationary measure for $ \SSS^*_1$.~To prove the uniqueness   and exponential mixing, we will  show that   the operator $\SSS_m^*$ is a contraction  if the space $X$ is endowed with an appropriate~metric and $m\ge1$ is sufficiently large.  The proof relies on the refined  uniform Feller property.  

\smallskip
 Let us endow $X$ with the metric $d_\theta$ given by
  $$
  d_\theta(u,v)= 1 \wedge (\theta \, d(u,v)),
  $$  where   $\theta>0$ is a large number that will be fixed later. We consider the Kantorovich metric on $\ppp(X)$ defined by 
   $$
  \| \sigma_1-\sigma_2\|_{K_\theta}=\sup_{L_\theta(f)\le 1} |\lag f,\sigma_1\rag-\lag  f,\sigma_2\rag|, \quad \sigma_1,\sigma_2\in \PP(X),
  $$where 
  $$
L_\theta(f)= \sup_{u,v\in X, u\neq v} \frac{|f(u)-f(v)|}{d_\theta(u,v)}.
$$
Proposition~\ref{P:2.4} follows immediately from the following two lemmas.   
\begin{lemma}\label{L:2.5}
For any $\theta\ge (\diam(X))^{-1}$ and   $\sigma_1, \sigma_2\in \mmm_+(X)$, we have
\be\label{E:2.15}
\f{1}{1+\theta}\|\sigma_1-\sigma_2\|_{K_\theta}\le \|\sigma_1-\sigma_2\|_{L}^*\le \diam(X)\|\sigma_1-\sigma_2\|_{K_\theta},
\ee where $\diam(X)=\sup_{u,v\in X}d(u,v).$
\end{lemma} 
 \begin{lemma}\label{L:2.6} For sufficiently large 
    number $ \theta>0$  and   integer~$m\ge1$, we have   
  \begin{equation}\label{E:2.16}
 	  \| \SSS^*_m\sigma_1-\SSS^*_m\sigma_2\|_{K_\theta}\le \frac12 \,  \| \sigma_1-\sigma_2\|_{K_\theta} \q\quad\text{for $\sigma_1,\sigma_2\in \PP(X)$.}
 \end{equation}
\end{lemma}

\bp[Proof of Lemma \ref{L:2.5}]
{\it Step 1}.
Let us prove the first inequality in \ef{E:2.15}. We take any~$f\in L_b(X)$ with $L_\theta(f)\le 1$. Replacing $f$ by $f-f(0)$, we may assume that $f(0)=0$. Then,   the definition of the metric $d_\theta$ implies that
 \begin{equation}\label{E:2.17}
 	\|f\|_X\le 1 \quad \text{and} \quad L(f)\le \theta,
 \end{equation}
where 
  $$
L(f)= \sup_{u,v\in X, u\neq v} \frac{|f(u)-f(v)|}{d(u,v)}.
$$  	 Thus $\|f\|_L\le 1+\theta$ and   we obtain the required inequality.

\smallskip
{\it Step 2.} Let us take any $f\in L_b(X)$ with $\|f\|_L\le 1$.
The second inequality in \ef{E:2.15} will be proved, if we show that $L_\theta(f)\le \diam(X)$. We claim that this is the case for $\theta\ge(\diam(X))^{-1}$. Indeed, since $L(f)\le \|f\|_L$, we have
$$
|f(u)-f(v)|\le d(u,v)  \le \diam(X) \wedge ( \diam(X) \theta d(u,v)  )\le  \diam(X) d_\theta(u, v)
$$
for any $u, v\in X$.  This completes the proof of Lemma \ref{L:2.5}.
  \ep
  
\begin{proof}[Proof of Lemma \ref{L:2.6}]  
{\it Step~1}.   Inequality \ef{E:2.16} will be established if we show that
for any $f\in L_b(X)$ with  $L_\theta(f)\le 1$, there is a function 
  $g\in L_b(X)$ with~$L_\theta(g)\le 1$ such that
$$
 |\lag \SSS_mf,\sigma_1\rag-\lag  \SSS_mf,\sigma_2\rag|\le \f{1}{2}|\lag g,\sigma_1\rag-\lag  g,\sigma_2\rag|, \q \sigma_1, \sigma_2\in\ppp(X)
$$
for some constant $\theta>0$ and integer $m\ge 1$ not depending on $f$.
As above, we may assume that $f$ vanishes at the origin. 
Note that the above inequality is trivially satisfied with  $g=2\SSS_m f$. Therefore, we only need to show that for an appropriate choice of $\theta$ and $m$ we have $L_\theta(g)\le 1$ or equivalently 
\be\label{E:2.18}
L_\theta(\SSS_m f)\le \f 1 2.
\ee

 {\it Step 2.}
By virtue of \ef{E:2.1},  there is a positive constant $C$     such~that
$$
 L(\SSS_m f) \le C\, \|f\|_\iin +\frac{1}{4} \,L(f), \q m\ge 1.
$$Using this inequality together with   
  \eqref{E:2.17}, we get
\be\label{E:2.19}
L(\SSS_m f) \le C   + \frac{\theta}{4} \le \f{\theta} {2},\q m\ge 1,
\ee
if $\theta\ge 4C$.   Further, 
thanks to \eqref{E:2.10}, we have
$$
  \sup_{u,v\in X}\| \SSS^*_k\delta_u-\SSS^*_k\delta_{v}\|_{L}^* \to 0 \q\quad \text{as $k\to \ty$}.
$$
Note that $\|f\|_L\le 1+\theta$. Therefore, we can find an integer $m=m(\theta)$ such that
$$
\sup_{u,v\in X}| \SSS_m f(u)-\SSS_m f(v)|\le \frac12.
$$
Combining this with \ef{E:2.19}, we arrive at \ef{E:2.18}
 \end{proof}


 \section{Proof  of Theorems \ref{T:1.1}-\ref{T:1.3}}\label{S:3}

   Before starting the proofs,  let us make a  few comments about the  assumptions entering these theorems. Hypothesis (D) ensures that the support~$\KK$ of the law of $\eta_1$  is contained in a Hilbert cube, so it is a compact set  in~$H$.  Since $S$ is locally Lipschitz in $H$, we have
   $$
   \|S(u_1)-S(u_2)\|\le C_R\|u_1-u_2\|, \q u_1, u_2\in B_R.
   $$
Combining this with \eqref{E:1.4},  we infer that the mapping $S:H\to H$ is  compact, i.e., the image under $S$ of any bounded   set is relatively compact.~Thus  the domain  of attainability~$\aA(B)$  from any bounded set~$B\subset H$ is compact.~By definition,~$\aA(B)$   is  invariant   for~\eqref{E:0.1}.  Therefore~$\aA(B_\rho)$ is a compact invariant absorbing set, where   $\rho$ is the number~in~\eqref{E:1.3}. 
   
   \medskip

   Let us give here the details of the proof of Theorem \ref{T:1.1}.~We apply  Theorem~\ref{T:2.1} for the generalised~Markov kernel   
$$
 P_1^V(u,\Gamma)=\E_u \left\{\ch_{\{u_1\in \Gamma\}} e^{V(u_1)}\right\}= \int_\Gamma  P_1(u, \!\dd  z) \,e^{V( z)} ,
$$
where $P_1(u,\Gamma)$ is the    transition function of   $(u_k,\pP_u)$,  $u\in \aA$, $\Gamma\in\BBBBB(\aA)$,  and~$V\in L_b(\aA)$.~Then Condition (i)     follows from Proposition \ref{P:4.1} applied for $B=\{0\}$, and~(ii) is proved in the  following lemma.   Applying the first assertion of Theorem~\ref{T:2.1}, we complete the proof of   Theorem \ref{T:1.1}. 
\begin{lemma}\label{L:3.1}
Under   Conditions {\rm(A)-(D)},
	for any $V\in L_b(H)$, $R>0$, and $r>0$, there is an integer $m\ge1$ and a number $p>0$ such that
$$
P_m^V(u,B_r(\hat u))\ge p\quad\mbox{for all $u\in \aA(B_R)$ and $\hat u\in \aA$}. 
$$
\end{lemma}
\begin{proof} As $V$ is bounded,  we have
  $$
   P_m^V(u,\Gamma)\ge e^{-m\|V\|_\ty} P_m(u, \Gamma), \quad \text{ for }u\in \aA(B_R), \, \Gamma\in\BBBBB(X).
  $$Using  \eqref{E:1.2}, the inclusion~$0\in \KK$ (see Condition~(D)), and a simple compactness argument, we can choose the numbers $r,p>0$ and the integer  $m\ge1$ such that
  $$
  P_m(u, B_r(\hat u))\ge p \quad \text{ for   $u\in X,$  $\hat u\in \aA.$}
  $$This implies the required result.     
   \end{proof}
    We establish Theorems~\ref{T:1.2} and~\ref{T:1.3}   in the following two sections.

    \subsection{The case of a potential with a small oscillation}
   
   Theorem \ref{T:1.2} is proved by applying the second assertion of Theorem~\ref{T:2.1} for the kernel $P_1^V$ in the compact  space $X=\aA(B_\rho)$.~Since $\aA$ is an  invariant set  for~\eqref{E:0.1}, we have $ P_1^V(u,X\setminus \aA)=0$ for $u\in \aA$.  Conditions (i) and~(ii) are established in Proposition \ref{P:4.1} and Lemma \ref{L:3.1}.~Thus Theorem \ref{T:1.2} will be proved if we check   (iii) and (iv). Indeed, by Theorem~\ref{T:2.1},   we will then have inequality \eqref{E:1.6}  for any  $u\in \aA(B_\rho)$, hence also for any~$u\in B_R$ by the absorbing property  (B). 
       
       \smallskip
       We shall prove (iii) and (iv) for a potential $V$ with a  sufficiently small oscillation.       Without loss of generality, we can always assume that $\lambda_V=1$. Indeed, it suffices to replace $V$ by $V-\log \lm_V$  (this   has no impact on the oscillation of~$V$).

     \subsubsection{Condition (iii)}
  
  \begin{lemma}\label{L:3.2}
  Under   Conditions {\rm(A)-(D)},
  there is  a number  $\De>0$ such that for any~$V\in L_b(H)$ with~$\Osc(V)\le \De$ and any $r>0$, we have  
 \begin{equation} \label{E:3.1}
\| P_k^V(\cdot, X\setminus \aA_r)\|_X  \to 0 \q\q\text{ as } k\to\iin.
\end{equation}
   \end{lemma}
\begin{proof} 
 Let us show that  if $\Osc(V)$ is sufficiently small, then     
\be\label{E:3.2}
\sup_{u\in X}   \E_u\left\{\ch_{\{u_k\notin \aA_r\}}\Xi^V_k \mek\right\}\le C(r, \rho) e^{-\al k/2}, \quad k\ge1,
\ee
where $\al>0$ is the number in   \eqref{E:0.2} and $\Xi^V_k$ is defined by \eqref{E:0.3}.
Indeed, let $f\in L_b(H)$ be a   non-negative function that
vanishes on $\aA$ and equals $1$ outside~$\aA_r$.
Observe that,  since~$\aA$ is invariant and compact, it contains\,\footnote{The irreducibility property on $\aA$ implies that $\aA=\supp \mu$.} the support of the unique stationary measure~$\mu\in \PP(H)$ of     $(u_k,\pP_u)$. 
 Thus  $\lag f, \mu\rag=0$. Note that
 \begin{equation}\label{E:3.3}
 	\text{as  $\lm_V=1$, we have  $\inf_{u\in\aA} V(u)\le0$ and $\|V\|_\ty\le \Osc(V)\le \delta$}.
 \end{equation}
 Combining this with~\eqref{E:0.2}, we~obtain        
      \begin{align*} 
   \E_u \!\left\{\ch_{\{u_k\notin\aaa_r\}} \Xi^V_k\mek  \right\}&\le  e^{ k  \|V\|_\ty}   \E_u  f( u_k) \le   C(r, \rho) e^{k\Osc(V)} e^{-\al k}. 
 \end{align*}
  Now assuming   $ \De\le\alpha/2$, we arrive at \ef{E:3.2}.      
	\end{proof}

  \subsubsection{Condition (iv)} 
  
  \begin{lemma}\label{L:3.3}
  Under   Conditions {\rm(A)-(D)},
there is  a number   $\De>0$ such that for any     $V \in L_b(H)$ with~$\Osc(V)\le~\De$,  
\begin{equation} \label{E:3.4}
\mM_R=\sup_{k\ge 1}  \|\PPPP^V_k\mek\|_{\aA(B_R)}  <\ty \q\q\text{ for any } R>0.
 \end{equation}
 \end{lemma}
\begin{proof}
{\it Step~1}. It suffices to prove the lemma for $R=\rho$. 
For any $\es>0$, let $\tau_\es $~be the first hitting time of the ball $B_\es$:
$$
\tau_\es=\min\{k\ge0: u_k\in B_\es\}.
$$ 
Let us show that for some $\De>0$, we have
\begin{equation}\label{E:3.5}
\sup_{u\in X}\E_ue^{\De \tau_\es}\le 2.
\end{equation} 
Indeed, by Conditions (A) and \rm{(D)},  there are  $q\in(0,1)$ and~$l\ge1$   such that 
$$
\pP_u\{u_l\in   B_\es\}\ge q, \quad u\in X.
$$Then for any   $k\ge1$, the Markov property gives   
\begin{align*}
\pP_u\{ k l<\tau_\es\} &\le \pP_u\{u_{jl}\notin B_\es, \, j=0, \ldots, k\}\\&\le(1-q)\pP_u\{u_{jl}\notin B_\es, \, j=0, \ldots, k-1\}\le (1-q)^k,\q u\in X,
\end{align*}
which allows to  conclude   that $\sup_{u\in X}\E_ue^{\De \tau_\es}$ is finite for some $\delta>0$. Choosing a smaller $\De$   and using the H\"older inequality, we obtain \ef{E:3.5}. 

\smallskip
{\it Step~2.}  Using \eqref{E:3.3} and the
     strong Markov property, we get for $u\in X$,
  $$
  \PPPP^V_k\mek (u)= \E_u\left\{\ch_{\{\tau\ge k\}} \Xi^V_k\mek \right\}+\E_u\left\{\ch_{\{\tau<k\}}  \Xi^V_k\mek\right\}\le \E_ue^{\De \tau_\es}+\E_u\left\{e^{\De \tau_\es}  \PPPP^V_k\mek (u_{\tau_\es})\right\}.
 $$
From \eqref{E:3.5} we derive    
\begin{equation}\label{E:3.6}
 \| \PPPP^V_k\mek\|_{X}\le 2+2 \|\PPPP^V_k \mek\|_{B_\es\cap X}.
 \end{equation}
   On the other hand, using inequality   $$
   |\PPPP_k^V f(v)-\PPPP_k^Vf(v')|  \le C \|f\|_L \, \|\PPPP_k^V{ \mek}\|_{\aaa(B)}\|v-v'\| 
   $$ 
   that follows from Proposition \ref{P:4.1}, with~$B= B_\rho$ and $f=\mek$, we get   $$
 \left|\PPPP_k^V\mek(u)-\PPPP_k^V\mek(0)  \right| \le \frac14 \|\PPPP_k^V\mek\|_{X}, \quad  u\in B_\es\cap X, \,k\ge1
 $$
 for  sufficiently small $\es>0$.
 Combining this with \eqref{E:3.6}, we infer
 $$
 \| \PPPP^V_k\mek\|_{X}\le 4+4\, \PPPP_k^V \mek(0)\le 4+4\|\PPPP^V_k\mek\|_\aaa.
$$
To conclude, it remains to recall that the sequence $\{\|\PPPP^V_k\mek\|_\aaa\}$   is bounded, by virtue of \ef{E:2.7}.
  \end{proof}

\subsection{The case of an arbitrary potential}
   
  As in the previous section,   to prove  Theorem \ref{T:1.3}, we need  to  check (iii) and~(iv) in the space $X=\aA(B_\rho)$. The arguments are more involved, since the oscillation of $V$ now  can be arbitrarily large.         
 
   \subsubsection{Condition (iii)}

  \begin{lemma}\label{L:3.4}
  Under   Conditions {\rm(A)-(E)},
  for any~$V\in L_b(H)$ and $r>0$, we have limit \eqref{E:3.1}.
  \end{lemma}
 \bp By Remark \ref{R:3.7}, there is   $\es\in(0, r)$   such that  the inclusion $\aaa(\aaa_\es)\subset\aaa_r$ holds. Note that this implies the following: if for some $k\ge 1$ and $\om\in\Omega$, $u_k(\om)\notin \aaa_r$, then $u_0(\om)\notin\aaa_\es, \ldots, u_{k-1}(\om)\notin\aaa_\es$, and since $\es<r$, we also have $u_k(\om)\notin\aaa_\es$. Let~$V_\es$ be a Lipschitz-continuous function that vanishes on $\aaa$ and coincides with~$V$ outside $\aaa_\es$. It follows that 
$$
\ch_{\{u_k\notin \aA_r\}}\Xi^V_k\mek=\ch_{\{u_k\notin \aA_r\}}\Xi^{V_\es}_k\mek.
$$
   Taking the expectation   and using the Cauchy--Schwartz inequality, we~get
  $$
\E_u\left\{\ch_{\{u_k\notin \aA_r\}}\Xi^V_k\mek\right\}=\E_u\left\{\ch_{\{u_k\notin \aA_r\}}\Xi^{V_\es}_k\mek \right\}\le \left(\E_u \ch_{\{u_k\notin \aA_r\}}\right)^{1/2}\left(\PPPP_k^{\mathbb{V}}\mek(u)\right)^{1/2},
  $$
where we set $\mathbb{V}=2V_\es$. Further, since the function $V_\es$ vanishes on $\aaa$, so does $\mathbb{V}$ and hence $\lm_{\mathbb{V}}=1$. In view of Lemma \ref{L:3.5},  
$$
 \|\PPPP_k^{\mathbb{V}}\mek\|_{X}\le \mM_\rho, \q k\ge 1.
$$
Let $f$ be a non-negative  Lipschitz-continuous function vanishing on $\aaa$ and  $1$ outside $\aaa_r$. Then, due to exponential mixing \eqref{E:0.2}, we have
$$
\sup_{u\in X}\E_u \ch_{\{u_k\notin \aA_r\}}\le \sup_{u\in X}\E_u f(u_k)\le C(r, \rho)e^{-\al k}, \q k\ge 1.
$$
Combining   last three inequalities, we arrive at \ef{E:3.1}.  
  
\ep

      \subsubsection{Condition (iv)}

\begin{lemma}
 \label{L:3.5}
 Under   Conditions {\rm(A)-(E)},
for any   $V\in L_b(H)$ and $R>0$, we have inequality  \eqref{E:3.4}.
\end{lemma}
\bp We shall use a bootstrap argument to establish this result.
Let  
$$
R_*=\sup\{R\ge0: \mM_R<\iin\}.
$$
The lemma will be proved if we show that
$R_*=\iin$.

\smallskip
{\it Step~1.} We first show that if $\mM_R$ is finite for some $R\ge0$, then so is $\mM_{R+\es}$ for some $\es\in(0,1)$. To this end, first note that in view of inequality \ef{E:4.1}, we have
 \begin{equation}\label{UFE}
|\PPPP_k^Vf(v)-\PPPP_k^Vf(v')|  \le C_R\, \|f\|_L \|\PPPP_k^V{ \mek}\|_{\aaa(B)}\|v-v'\|
\end{equation}
for any $B\subset B_{R+1}$ and $v, v'\in \aaa(B)$. Applying  inequality~\ef{UFE}  with    $f=\mek$ and $B= B_{R+\es}$, we get
\begin{align*}
|\PPPP_k^V\mek(u)|&\le |\PPPP_k^V\mek(v)|+ C_R\|\PPPP_k^V\mek\|_{\aaa(B_{R+\es})}\|u-v\|, \q u, v\in \aaa(B_{R+\es}).
\end{align*}
In particular, this inequality is true for any $v\in \aaa(B_{R})$ and  $u\in  \aaa(B_{R+\es})$. Therefore, taking first the infimum over $v\in\aaa(B_{R})$ and then supremum over $u\in\aaa(B_{R+\es})$, we derive
\begin{align} 
\|\PPPP_k^V\mek\|_{\aaa(B_{R+\es})}&\le \mM_{R}+C_R\|\PPPP_k^V\mek\|_{\aaa(B_{R+\es})}\sup_{u\in\aaa(B_{R+\es})}\inf_{v\in\aaa(B_{R)}}\|u-v\|\nonumber\\
&=  \mM_{R}+C_R\|\PPPP_k^V\mek\|_{\aaa(B_{R+\es})}\, d_H(\aaa(B_{R+\es}), \aaa(B_{R})),
\label{E:3.7}
\end{align}
 where $d_H(E, F)$ is the Hausdorff distance between the   sets $E,F\subset H$. We use the following result proved in the next section.
\begin{lemma}\label{L:3.6} 
For any $R\ge0$,  we have  $ d_H(\aaa(B_{R+\es}), \aaa(B_R))\to 0$ as $\es\downarrow 0$. Moreover, if $R>0$, we have also
  $ d_H(\aaa(B_{R-\es}),  \aaa(B_R))\to 0$ as $\es\downarrow 0$. 
\end{lemma}

In view of the first assertion of this lemma, 
$$
d_H(\aaa(B_{R+\es}), \aaa(B_{R}))\le \frac 1{2C_R}
$$for $\es>0$ sufficiently small.
  Combining this with   \eqref{E:3.7}, we get
$$
\|\PPPP_k^V\mek\|_{\aaa(B_{R+\es})}\le 2\,\mM_R, \q k\ge 1,
$$
which implies    $\mM_{R+\es}\le 2\,\mM_R<\infty$.

\smallskip
{\it Step~2.} In this step, we show that $R_*=\infty$. Note that for~$R=0$, we have~$\aA(B_R)=\aA$, so that $\mM_0$ is finite in view of  \eqref{E:2.7}. The result of the previous step implies that $R_*>0$ and
  if $R_*<\iin$, then it cannot be attained, i.e., $\mM_{R_*}=\iin$. In search of a contradiction, assume that $R_*<\iin$ and take any~$\es\in (0, R_*)$.  As above,   we apply   inequality \ef{UFE}  with $f=\mek$ and~$B=  B_{R_*} $:
$$
|\PPPP_k^V\mek(u)| \le |\PPPP_k^V\mek(v)| + C_{R_*}\|\PPPP_k^V\mek\|_{\aaa(B_{R_*})}\|u-v\|, \q u, v\in \aaa(B_{R_*}).
$$
Taking first the infimum over $v\in\aaa(B_{R_*-\es})$ and then the supremum over $u\in\aaa(B_{R_*})$, we obtain
$$
\|\PPPP_k^V\mek\|_{\aaa(B_{R_*})}\le   \mM_{R_*-\es}+C_{R_*}\|\PPPP_k^V\mek\|_{\aaa(B_{R_*})}d_H(\aaa(B_{R_*}),\aaa(B_{R_*-\es})).
$$
 Using the second assertion of  Lemma \ref{L:3.6}, for sufficiently small~$\es>0$, we get  
 $$
 d_H(\aaa(B_{R_*}),\aaa(B_{R_*-\es}))\le \frac 1{2C_{R_*}}.
 $$ We thus infer
$$
\|\PPPP_k^V\mek\|_{\aaa{(B_{R_*})}}\le 2\,\mM_{R_*-\es}, \q k\ge 1.
$$
This contradiction proves that $R_*=\infty$.
\ep

\subsubsection{Proof of Lemma   \ref{L:3.6}}
The second part of the lemma readily follows from the definition of $\aaa(B_R)$ and its compactness. The first one is more delicate, and this is where we use Condition (E). Clearly, it is sufficient to show that
\begin{equation}\label{E:3.8}
	 \aaa(B_R)= \bigcap_{\es>0} \aaa(B_{R+\es}) .
\end{equation}
The proof of this equality is divided into two steps.

\smallskip
{\it Step 1: Reduction.}  Without loss of generality, we can assume that $R$ is smaller than the number $\rho$ in (B).  
For any~$r>0$, we  introduce the set
$$
 \ccc_r=\{u\in Y:  d' (u,\aaa(B_R))\le r\},  
$$ where  $d'$ is the metric in   (E) and   $Y=\aaa(B_{\rho+1})$.
 We claim that \eqref{E:3.8} will be established if we show that for any $r>0$ there is $\es>0$ such that  
\be\label{E:3.9}
\aaa(m,  B_{R+\es})\subset \ccc_r \quad \text{for any $m\ge1$.}
\ee  
Indeed, once this is proved, we will have 
\be\label{E:3.10}
\bigcup_{m=1}^\iin \aaa(m,  B_{R+\es})\subset \ccc_r.
\ee
Now note that the set $\ccc_r$ is closed in $H$ with respect to the natural topology. Indeed, if the sequence $\{u_k\}\subset  \ccc_r$ converges to $u$ in $H$,  then applying the triangle inequality,  we obtain  
$$
 d'(u, \aaa(B_R))\le d'(u, u_k)+d'(u_k, \aaa(B_R))\le d'(u, u_k)+r.
$$
Letting $k$ go to infinity and using the fact that the convergence in $H$ implies the one in $d'$, we get $u\in \ccc_r$. Therefore, taking the closure in $H$ in the inclusion~\ef{E:3.10}, we see that 
$   \aaa(B_{R+\es})\subset \ccc_r$. Letting $r$ go to zero, we arrive at~\eqref{E:3.8}.

\smallskip
{\it Step 2: Derivation of \ef{E:3.9}}. Let us fix any $r>0$. First note that, since the
topology of $d'$ is weaker than the natural one of $H$,
   for any~$u\in Y$, there is~$a>0$ such that $d'(u, v)\le  r$, provided $\|u-v\|\le  a$. Using the compactness of~$Y$, we see  that $a$ can be taken uniformly for~$u\in Y$.  
 Let us show that \ef{E:3.9} holds for sufficiently small $\es>0$. Indeed, take any~$m\ge1$ and~$u_*\in\aaa(m,B_{R+\es})$. Clearly,~$u_*\in Y$ if $R+\es\le \rho+1$.  To show that~$d'(u_*, \aA(B_R))\le r$, note that  there are $u_0\in B_{R+\es}$ and~$\eta_1, \ldots, \eta_m\in \kkk$ verifying 
$$
u_1=S(u_0)+\eta_1, \ldots, u_m=S(u_{m-1})+\eta_m
$$with $u_*=u_m$.
Let us take any $v_0\in B_R$ such that $\|u_0-v_0\|\le \es$ and define 
$$
v_1=S(v_0)+\eta_1, \ldots, v_m=S(v_{m-1})+\eta_m. 
$$Using   the translation invariance of $d'$ and  inequality \ef{E:1.7}, we obtain
$$
d'(u_m, v_m)=d'(S(u_{m-1}),S(v_{m-1}))\le d'(u_{m-1},v_{m-1}).
$$
Iterating this and using $u_*=u_m$, we arrive at
$$
d'(u_*, v_m)\le d'(u_1,v_1).
$$But for sufficiently small $\es$ we have 
$$
\|u_1-v_1\|=\|S(u_0)-S(v_0)\|\le a.	
$$ Thus   $d'(u_1,v_1)\le r$, by definition of $a$.~This implies that~$d'(u_*,v_m)\le r$, and
to conclude, it remains to note that $v_i$ all belong to~$\aaa(B_R)$ by definition. 

\medskip
\begin{remark}\label{R:3.7} Literally repeating the argument of the proof of \eqref{E:3.8}, we get 
$$
\bigcap_{\es>0}\aaa(\aaa_\es)=\aaa.
$$
This implies that for any $r>0$, there is $\es>0$ such that 
 $\aaa(\aaa_\es)\subset\aaa_r$.
 
\end{remark}

 \section{Refined uniform Feller property}\label{S:4}
 This section is devoted to the proof of the following result.
 \begin{proposition}\label{P:4.1}
  Under Conditions {\rm(A)-(D)},
 for any  $V\in L_b(H)$,~$R>0$, and~$c\in(0,1)$, there is a number~$C=C(\|V\|_L,R,c)>0$     such~that
\begin{equation}
|\PPPP_k^Vf(v)-\PPPP_k^Vf(v')|  \le \left( C\, \|f\|_\iin +c^k \, \|f\|_L\right) \|\PPPP_k^V{ \mek}\|_{\aaa(B)}\|v-v'\|\label{E:4.1} 
\end{equation} 
for any  set $B\subset B_{R}$,  initial points~$v, v'\in \aaa(B)$,  and function  $f\in L_b(H)$.
\end{proposition}

We prove this proposition by developing
  the ideas of the proof of the uniform Feller property on $\aA$ established in Theorem~3.1 in \cite{JNPS-2012}.
We start by recalling the properties of the coupling process.~Let~$\PPP(v)$ be the law of the trajectory~$\{u_k\}$ for~\eqref{E:0.1} issued from~$v\in\aA(B)$, i.e., $\PPP(v)$ is a probability measure on the direct product  of countably many copies of~$\aA(B)$.  
  The following result is a version of Proposition~3.2 in \cite{JNPS-2012}; see   Section~3.2.2 in~\cite{KS-book} for~the~proof. 
\begin{proposition} \label{P:4.2} For   sufficiently large integer $N\ge1$ there is a probability space $(\Omega_N,\FF_N,\IP_N)$  and an $\aA(B)\times\aA(B)$-valued Markov process $(u_k,u_k')$ on~$\Omega_N$ parametrised by the initial   point $(v,v')\in \aA(B)\times\aA(B)$ for which the following properties hold. 
\begin{itemize}
\item[\bf(a)] 
The $\IP_N$-laws of the sequences $\{u_k\}$ and~$\{u_k'\}$ coincide with~$\PPP(v)$ and~$\PPP(v')$, respectively.
\item[\bf(b)]  The projections ${\mathsf Q}_N (u_k-S(u_{k-1}))$ and ${\mathsf Q}_N (u_k'-S(u_{k-1}'))$ coincide for all~$\omega\in \Omega_N$.
\item[\bf(c)] There is a   number $C_N>0$ such that
for any integer $r\ge1$, we~have\,\footnote{The relation ${\mathsf P}_Nu_k={\mathsf P}_Nu_k'$ in \eqref{E:4.2} should be omitted for $r=1$.} 
\begin{equation}
\IP_N\bigl\{{\mathsf P}_Nu_k={\mathsf P}_Nu_k' \mbox{ for $1\le k\le r-1$}, {\mathsf P}_Nu_{r}\ne{\mathsf P}_Nu_{r}'\bigr\}
\le C_N\gamma_N^{r-1}\|v-v'\|,\label{E:4.2}
\end{equation}
where $ \gamma_N$ is the number in Condition {\rm (C)},  ${\mathsf P}_N$ is the orthogonal projection onto       $\textup{span}\{e_1,\dots,e_N\}$ in $H$ and ${\mathsf Q}_N=1- {\mathsf P}_N$.

 \end{itemize}
\end{proposition}

\begin{proof}[Proof of Proposition~\ref{P:4.1}]
  Without loss of generality, we can assume  that $f $ and~$V$ are non-negative functions on $ \aA(B)$.

 \smallskip
 
 Let us fix an initial point $(v,v')\in\aA(B)\times \aA(B)$ such that~ $\varkappa:=\|v-v'\|\le 1$, a sufficiently large integer  $N\ge1$, and  apply Proposition~\ref{P:4.2}.~Let $(u_k,u_k')$ be the corresponding sequence. We denote by~$A(r)$ the event on the left-hand side of~\eqref{E:4.2}, and 
$$\tilde A(r)=\bigl\{{\mathsf P}_Nu_k={\mathsf P}_Nu_k'\mbox{ for   $1\le k\le r$} \bigr\}.
$$ Then we have 
\begin{equation} \label{E:4.3}
\PPPP_k^Vf(v)-\PPPP_k^Vf(v')=\sum_{r=1}^kI_k^r+ \tilde I_k,
\end{equation}
where   
\begin{align*}  
I_k^r&=\E_N \bigl\{\ch_{A(r)}
\bigl(\Xi^V_kf(u_k)-\Xi^V_kf(u_k')\bigr)\bigr\},\\
\tilde I_k&=\E_N \bigl\{\ch_{\tilde A(k)}
\bigl(\Xi^V_kf(u_k)-\Xi^V_kf(u_k')\bigr)\bigr\},
\end{align*}
and~$\E_N$ is the expectation corresponding to~$\IP_N$.

 \smallskip
{\it Step~1. Estimate for $I_k^r$.}
Let us show that
\begin{equation}
| I^r_{k}|\le C_{N}\|f\|_\ty \gamma_N^{r-1}e^{r\|V\|_\infty} \|\PPPP_{k}^V\mek\|_{\aA(B)}\,\varkappa,  \q\text{   $1\le r\le k$}.\label{E:4.4}
\end{equation}   
Indeed,  let~$\FF_k^N$ be the filtration generated by~$(u_k,u_k')$. Taking the conditional expectation given~$\FF_r^N$, using the fact that  $f $ and $V$ are non-negative, and carrying out some simple estimates, we~derive
\begin{align*} 
I_k^r
\le\E_N\bigl\{\ch_{A(r)}\Xi^V_kf(u_k)\bigr\}&\le\|f\|_\ty e^{r\|V\|_\infty}
\E_N\bigl(\ch_{A(r)} \PPPP_{k-r}^V{\mek} (u_r)\bigr)\notag\\
&\le \|f\|_\ty e^{r\|V\|_\infty}\|\PPPP_{k}^V\mek\|_{\aA(B)}\IP_N\bigl\{A(r)\bigr\}.
\end{align*}
Using~\eqref{E:4.2}, we obtain \eqref{E:4.4}.

\smallskip
{\it Step~2. Squeezing.} Before estimating $\tilde I_k$, let us show the following squeezing property   on the event $\tilde A(k)$:
\begin{equation}\label{E:4.5}
\|u_{r}-u_{r}'\|
\le  \gamma_N^r\, \varkappa, \quad 1\le r\le k.
\end{equation} Indeed, 
using property (b) in Proposition \ref{P:4.2}, we get   
$$
\|u_{r}-u_{r}'\|= \|{\mathsf Q}_N(u_{r}-u_{r}')\|=\|{\mathsf Q}_N(S(u_{r-1})-S(u_{r-1}'))\| \le \gamma_N\, \|u_{r-1}-u_{r-1}'\|.
$$Iterating this, we arrive at the required result.
 
 \smallskip
{\it Step~3.~Estimate for $\tilde I_k$.}
Let us show that  
\begin{equation}\label{E:4.6}
|\tilde I_{k}|\le C_1 \, (\gamma_N \,\| f\|_{\infty}+\gamma_N^k \,\| f\|_{L}) \, \|\PPPP_k^V{\mek}\|_{\aA(B)} \, \varkappa 
\end{equation} for some number $C_1=C_1(\|V\|_L)>0$ not depending on $N$.
Indeed, 
\begin{align} \label{E:4.7}
\tilde I_k
&=\E\bigl\{\ch_{\tilde A(k)} \Xi^V_k \mek (u_k)[ f (u_k)-f (u_k')] \bigr\}\nonumber\\&\quad+\E\bigl\{\ch_{\tilde A(k)}[ \Xi^V_k \mek (u_k)- \Xi^V_k \mek  (u_k')]f (u_k') \bigr\}=: J_{1,k}+J_{2,k}.
\end{align}     We derive from~\eqref{E:4.5},
$$
|J_{1,k}|\le   \E\bigl\{\ch_{\tilde A(k)}  \Xi^V_k{\mek} (u_k) | f (u_k)-f (u_k')| \bigr\} \le \gamma_N^k \, \| f\|_{L}     \|\PPPP_k^V{1}\|_{\aA(B)}\,\varkappa.
$$ Similarly, as $V\in L_b(H)$,
\begin{align*}
|J_{2,k}|&\le \|f\|_\ty \E\bigl\{\ch_{\tilde A(k)}| \Xi^V_k \mek(u_k)- \Xi^V_k{\mek}(u_k')| \bigr\}\\&\le \|f\|_\ty  \E\left\{\ch_{\tilde A(k)}\Xi^V_k\mek(u_k)\left[ \exp\left(\sum_{n=1}^k|V(u_n)-V(u_n')| \right)-1 \right] \right\}\\&\le \|f\|_\ty    \left[ \exp\left(  \varkappa  \gamma_N \,(1-\gamma_N)^{-1} \|V\|_{L}   \right)-1 \right] \|\PPPP_k^V{\mek}\|_{\aA(B)}.
\end{align*}
Combining the   estimates for $J_{1,k}$ and $J_{2,k}$ with \eqref{E:4.7}, we obtain \eqref{E:4.6}.

\smallskip
{\it Step~4.}
Substituting \eqref{E:4.4} and  \eqref{E:4.6}    into~\eqref{E:4.3}, we derive 
\begin{align*}
&|\PPPP_k^Vf(v)-\PPPP_k^Vf(v')|\\
&\q\q\q\q\le \left(  \tilde C_N\, \|f\|_\ty\sum_{r=1}^k  \gamma_N^{r-1}e^{r\|V\|_\infty}
+ C_1 \gamma_N^k \, \|f\|_L \right)   \|\PPPP_k^V{\mek}\|_{\aA(B)} \, \varkappa\\
&\q\q\q\q\le\left( C\, \|f\|_\iin +c^k \, \|f\|_L\right) \|\PPPP_k^V{ \mek}\|_{\aaa(B)}\,\varkappa 
\end{align*}
 for sufficiently large $N$.
\end{proof}

\section{Applications}\label{S:5}
 In this section, we present various   corollaries of Theorems \ref{T:1.1}-\ref{T:1.3}.

\subsection{Existence and analyticity of  the pressure function} 
We start with the existence of the pressure function. 
\begin{proposition}\label{P:5.3}
Assume that Conditions {\rm(A)-(D)} are fulfilled. Then the following limit  (called pressure function) exists
\be\label{E:5.9}
Q(V, u)=\lim_{k\to\iin}\f{1}{k}\log  \PPPP^V_k \mek (u)
\ee
 for any $V\in C (H)$ and $u\in H$. Moreover, this limit does not depend on~$u$  if we have one of the following properties:  
\begin{enumerate}
\item[\bf(1)]	The initial condition $u$ belongs to $\aA$.
\item[\bf(2)]  $\Osc(V)\le \delta$, where $\delta$ is the number in Theorem \ref{T:1.2}.
 \item[\bf(3)] Condition {\rm(E)} is satisfied.
 \end{enumerate} 
 The limit in \eqref{E:5.9} is denoted by $Q(V)$  if one of the properties {\rm(1)-(3)} is~satisfied.
\end{proposition}
\begin{proof}  First assume that $V\in L_b(H)$.
If we have  one of     (1)-(3), then by \eqref{E:1.6},  
$$
 \lambda_V^{-k}  \PPPP^V_k \mek(u) \to h_V(u) \quad\text{as $k\to \ty $}.
 $$
Taking the logarithm, we infer that $Q(V,u)=\log \la_V $.  In the general case, we cannot use the multiplicative ergodicity, so we proceed differently.  We use the following lemma which is established below.
\begin{lemma} \label{L:5.4} Under   Conditions {\rm(A)-(D)}, for any $V\in L_b(H)$ and     bounded set~$B\subset H$, there is a number~$C>0$ such that
\begin{equation}
 C^{-1} \|\PPPP^V_k \mek\|_{B}\le  \|\PPPP^V_k \mek\|_{\aA(B)}  \le C \|\PPPP^V_k \mek\|_{B}, \quad\,\,  k\ge1.\label{E:5.10'}
\end{equation}  
\end{lemma}
We take any $u\in H$ and apply \eqref{E:5.10'} for~$B=\{u\}$,
\begin{equation}
 C^{-1}\, \PPPP^V_k \mek (u)\le  \|\PPPP^V_k \mek\|_{\aA(\{u\})}  \le C \,\PPPP^V_k \mek (u), \quad\,\,  k\ge1. \label{E:5.10}
\end{equation}  Since the set $\aA(\{u\})$ is invariant for \eqref{E:0.1},  we have 
$$
\|\PPPP^V_{n+m} \mek\|_{\aA(\{u\})}  \le \|\PPPP^V_{n} \mek\|_{\aA(\{u\})} \|\PPPP^V_{m} \mek\|_{\aA(\{u\})}, \quad m,n\ge1.
$$This implies that the function $f:\Z_+\to \rr$, $f(k)=\log \|\PPPP^V_k \mek\|_{\aA(\{u\})}$ is sub-additive, hence, by the Fekete lemma, the following limit exists
$$
\lim_{k\to\ty} \frac1k\log \|\PPPP^V_k\mek\|_{\aA(\{u\})}.
$$
 Applying \eqref{E:5.10}, we get the existence of limit~\eqref{E:5.9}.  
 
 \smallskip
  Now let us assume   that  $V\in C(H)$. As $\aA(\{u\})$ is compact in $H$, we can find a sequence $V_n\in L_b(H)$ such that $\|V-V_n\|_{\aA(\{u\})}\to0$ as~$n\to\ty$. Then using the inequality 
$$
\left|\frac1k \log \PPPP_k^V \mek(u)-\frac1k \log \PPPP_k^{V_n} \mek(u)\right|\le \|  V-  V_n\|_{\aA(\{u\})}, \quad k,n\ge1,
$$we get the existence of limit~\eqref{E:5.9} for any $V\in C(H)$ and $u\in H$.
\end{proof}

\begin{proof}[Proof of Lemma \ref{L:5.4}] Using the Markov property and the fact that $V$ is bounded on $\aA(B)$, we get
\begin{align*}
  \PPPP^V_k \mek (u)&\le C_1\,\E_u \left( \PPPP^V_{k-1} \mek(u_1)\right)\le C_1\,  \|\PPPP^V_{k-1} \mek\|_{\aA(1,B)}   \le C_2\,  \|\PPPP^V_k \mek\|_{\aA(B)}  \end{align*}  
  for any $k\ge1$ and  $u\in B$. This proves the first inequality in \ef{E:5.10'}. 
    The proof of the second inequality relies on the uniform Feller property.
We argue
by contradiction.  If this inequality is not true, then there is a sequence $k_n\to\ty$ such that 
  \begin{equation}\label{E:5.12}
  \frac{ \|\PPPP^V_{k_n} \mek\|_{B} }{\|\PPPP^V_{k_n} \mek\|_{\aA(B)} }\to 0\quad\text{as $n\to\ty$}.
  \end{equation}     
  By Proposition~\ref{P:4.1},  the sequence $\{\|\PPPP^V_k \mek\|_{\aA(B)}^{-1}\PPPP^V_k\mek,k\ge0\}$ is uniformly equicontinuous on~$\aA(B)$.      The Arzel\`a--Ascoli theorem implies the existence of a subsequence of $k_n$, which is again denoted by $k_n$,   and a non-negative function~$g\in C(\aA(B))$ such that 
   \begin{equation}\label{E:5.13}
     \frac{\PPPP^V_{k_n}\mek  }{ \|\PPPP^V_{k_n} \mek\|_{\aA(B)}} \to g\quad\mbox{in $C(\aA(B))$ as $n\to\ty$}.
  \end{equation} 
  Clearly, $\|g\|_{\aA(B)}=1$. Hence there is an integer $m\ge1$ and a point   $v_*\in \aA(m,B)$ such that   $g(v_*)>0$.
 From~\eqref{E:5.13} and the Lebesgue
theorem on dominated convergence it follows  that 
  $$          \frac{\PPPP^V_{k_n+m}\mek  }{ \|\PPPP^V_{k_n} \mek\|_{\aA(B)}} \to \PPPP^V_{m} g\quad\mbox{in $C(\aA(B))$ as $n\to\ty$},
  $$where $\PPPP^V_{m} g(u_*)>0$   for some $u_*\in B$. Therefore, for any sufficiently large $n\ge1$, 
  $$
  \frac{\PPPP^V_{k_n}\mek  (u_*)}{ \|\PPPP^V_{k_n}\mek\|_{\aA(B)}}  \ge e^{-m \|V\|_{\aA(B)}}  \frac{\PPPP^V_{k_n+m}\mek(u_*)  }{ \|\PPPP^V_{k_n} \mek\|_{\aA(B)}}\ge \frac12 e^{-m \|V\|_{\aA(B)}} \PPPP^V_{m} g(u_*) ,
  $$ 
which contradicts \ef{E:5.12} and proves \eqref{E:5.10'}.  
\end{proof}

Combining  convergence   \eqref{E:1.6} with a well-known   perturbation argument~\cite{wu-2001, MT-2005}, we 
   prove the  analyticity of the pressure function.
\begin{proposition}\label{AP:5.3}
Assume that Conditions {\rm(A)-(D)} are fulfilled and $V\in L_b(H)$. Then there is a number $p>0$ such that the following limit exists
\be\label{AE:5.9}
Q(z V)=\lim_{k\to\iin}\f{1}{k}\log  \PPPP^{z V}_k \mek (u)
\ee
 for any $u\in H$ and    $z\in D_{p}=\{x\in  \C: |x|\le p\}$. Moreover,        the map 
$
z \mapsto Q(z V)
$ is real-analytic in some neighborhood of the origin. If one of the properties {\rm(1)} and {\rm(3)} in Proposition \ref{P:5.3} is satisfied, then this map   exists and  is  real-analytic on~$\rr$.
\end{proposition}

\begin{proof}Let us denote by $L_{b,\C}(H)$ the complexification of the space $L_{b}(H)$ and by~$\LL$ the space of bounded linear operators from~$L_{b,\C}(H)$ to~$L_{b,\C}(H)$ endowed with the natural norm $\|\cdot \|_\LL$. We
consider the family $\{\PPPP^{z V}_1:  z\in D_p\}$   in~$\LL$.~It is straightforward to check that this is a {\it holomorphic family}   in the sense of Section~VII.1.1 in~\cite{kato-95}, p. 365. By   exponential mixing~\eqref{E:0.2}, the   operator $\PPPP_1=\PPPP^{z V}_1$ with $z=0$ has a simple isolated eigenvalue~$\lambda_{0}=1$ corresponding to   eigenvectors~$h_0=\mek$ and~$\mu_0=\mu$.
Let~${\mathsf P}_0f=\lag f, \mu\rag$ be the spectral projection associated with this eigenvalue.~Clearly,   the spectral radius of the operator $\PPPP_1(1-{\mathsf P}_0)$ is less than $e^{-\alpha}<1$. 

\smallskip

 By Kato's holomorphic perturbation theorem (see Theorems~1.7 and~1.8 in Section~VII.1.3 in~\cite{kato-95}, p. 368--370), there is a number $p>0$ such that    the following property holds:
 \begin{enumerate}
 \item[\bf $\bullet$]	 The operator  $\PPPP_1^{zV}$   has a simple eigenvalue $\lambda_{zV}$ for any $z\in D_p$.~Moreover, the maping $z\to (\lambda_{zV}, {\mathsf P}_{zV})$  is analytic on $D_p$, where ${\mathsf P}_{zV}$ is the spectral projection associated with $\lambda_{zV}$. 
 \end{enumerate}
 In particular, for sufficiently small $p>0$,   we have  	  
 \begin{gather}
 	\inf_{\, z\in D_p} |\lambda_{zV}|>   \gamma:=(e^{-\alpha}+1)/2,\label{EE1}\\
 	 \sup_{\, z\in D_p}\|\PPPP_1^{zV}(1-{\mathsf P}_{zV})\|_\LL \le \gamma,\label{EE2}\\
 	M  :=\sup_{x\in S_\gamma, \, z\in D_p} \|(I-x \,\PPPP_1^{zV}(I-{\mathsf P}_{zV}))^{-1}\|_{\LL}<\ty,\label{EE3}
  \end{gather}
 where $S_\gamma=\{x\in \C: |x|=\gamma^{-1}\}$.
  By the  Cauchy integral formula, we have
 \begin{align*}
 	(\PPPP_1^{zV}(I-{\mathsf P}_{zV}))^k&=\frac{1}{k!}\frac{\p^k}{\p x^k}  (I-x \,\PPPP_1^{zV}(I-{\mathsf P}_{zV}))^{-1}|_{x=0}\\
 	&=\frac{1}{2\pi i} \int_{S_\gamma} x^{-k-1} (I-x \,\PPPP_1^{zV}(I-{\mathsf P}_{zV}))^{-1} \dd x, \quad k\ge 1.
 \end{align*}
Combining this with \eqref{EE2} and \eqref{EE3}, we see that
$$
\|\PPPP_k^{zV}  - \lambda_{zV}^k {\mathsf P}_{zV}\|_\LL=\| (\PPPP_1^{zV}(I-{\mathsf P}_{zV}))^k\|_\LL\le M\gamma^k .
$$
This and \eqref{EE1} readily imply limit \eqref{AE:5.9}. 

\smallskip
The remaining assertions are proved  using limit \eqref{E:1.6} for  the  potential~$z_0 V$ and applying a similar perturbation argument for  the family $\{\PPPP^{z V}_1:  z\in D_p(z_0)\}$,  where $z_0\in \rr $ is arbitrary and $D_p(z_0)=\{x\in  \C: |x-z_0|\le p\}$. 
\end{proof}

In view of limit \eqref{AE:5.9}, we can apply Bryc's criterion (see Proposition 1 in~\cite{bryc-1993}). 
We obtain immediately that for any $V\in L_b(H)$ with $\lag V,\mu\rag=0$ and any $u\in H$, the following   central limit theorem holds
 $$
 \DD_u\left(\frac{1}{\sqrt k}\sum_{n=1}^kV(u_n)\right)\to N(0,\sigma_V), \quad k\to \ty,
 $$
	where 
	$\mu$ is the  stationary measure of $(u_k, \pP_u)$, $\DD_u$ is  the distribution of a random variable under the law $\pP_u$, and $\sigma_V= \frac{\p^2}{\p\alpha^2} Q(\alpha V)|_{\alpha=0}$.
  See Section 4.1.3 in \cite{KS-book} for another proof of this result and~\cite{shirikyan-ptrf2006} for an estimate for the rate of convergence.

\subsection{Large deviations} \label{S:5.4}

In this section, we give some applications to large deviations principle (LDP). We use some standard terminology from the LDP theory (e.g., see \cite{ADOZ00, DS1989}). Recall that the {\it occupation measures} for  the trajectories of \eqref{E:0.1} are defined by
$$
\zeta_k=\f{1}{k}\sum_{n=0}^{k-1}\De_{u_n}.
$$ 
In Theorem 1.3 in \cite{JNPS-2012}, a level-2 LDP is obtained    
 for the family  $\{\zeta_k\}$     in the case when the initial condition belongs to   $\aA$.~In this section, we complete that theorem, by stating two   results that establish LDP   in the case of an arbitrary   initial condition~in~$H$. 
The following theorem gives, in particular, a level-1 LDP of local type
 under the same conditions as in~\cite{JNPS-2012}.
     \begin{theorem}  Let Conditions {\rm(A)-(D)} be fulfilled. Then       for any  non-constant   function $f\in L_b(H)$,  there is $\es=\es(f)>0$ 
 and a convex function~$I^f:\rr\to \rr_+$
such that,  
  for any     open subset~$O$ of the interval $(\langle f, \mu\rangle-\es, \langle f, \mu\rangle+\es)$ and $u\in H $, we have
$$
\lim_{k\to\infty} \frac1k\log   \pP_u\left\{\lag f,\zeta_k\rag \in O\right\}= -\inf_{x\in O} I^f(x),
$$  where $\mu$ is the  stationary measure.   This  limit   is uniform with respect to~$u$ in a bounded set of  $H $.  Moreover, if Condition~\rm{(E)} is also fulfilled, then $\es=+\ty$.
       \end{theorem} 
        This theorem   follows immediately from the analyticity of the pressure function established in Proposition \ref{AP:5.3} and a local version of the G\"artner--Ellis theorem (e.g., see Theorem A.5 in~\cite{JOPP}).
     
     \smallskip
    A level-2   LDP holds   in the whole space     $H$, 
      provided that Conditions~{\rm(A)-(E)} are   fulfilled. Namely, we have the following result.
       
        \begin{theorem}   Let the assumptions {\rm(A)-(E)} be fulfilled. Then, there is a convex function $I:\ppp(H)\to [0, +\iin]$ with compact level sets $\left\{I\le M\right\}$ in~$H$ for any~$M>0$  and that is infinite outside $\ppp(\aaa)$ such that for any random initial point $u_0$ whose law $\lm=\DD u_0$ has a bounded support in $H$, we have
$$
-\inf_{\sigma\in \dot\Gamma}I(\sigma)\le\liminf_{k\to \iin}\f 1 k\log \pp_\lm\{\zeta_k\in \Gamma\}\le\limsup_{k\to \iin}\f 1 k\log\pp_\lm\{\zeta_k\in \Gamma\}\le -\inf_{\sigma\in \bar\Gamma}I(\sigma)
$$       
for any subset $\Gamma\subset \ppp(H)$, where $\dot\Gamma$ and $\bar\Gamma$ stand for its interior and closure, respectively. Moreover, the function $I$ can be written as
\be\label{e57.3}
I(\sigma)=\sup_{V\in C(\aaa)}\left(\lag V, \sigma\rag-Q(V)\right),      \q\q\sigma\in\ppp(\aaa),
\ee
where $Q(V)$ is the   pressure function defined in  Proposition \ref{P:5.3}. 
\end{theorem}
    This result  can be proved   using Theorem \ref{T:1.3} and literally repeating the arguments of the proof of Theorem~1.3 in \cite{JNPS-2012} based on the application of   the Kifer's criterion obtained in~\cite{kifer-1990}. 


\subsection{The SLLN time}

In   paper \cite{shirikyan-ptrf2006}, a   strong law of large numbers   is obtained  for system \eqref{E:0.1}. More precisely, it is proved that for any~$f\in L_b(H)$,   $u\in H$, and   $\es>0$, the following inequality holds  
$$
\left|\f{1}{k}\sum_{n=1}^k f(u_n)-\lag f, \mu\rag\right|\le C\, k^{-1/2+\es} \q\q\text{ for } k\ge T,
$$
where  $\mu$ is the stationary measure of $(u_k, \pp_u)$ and 
$T\ge1$ is a random integer    whose any   moment is finite, i.e., $\e_uT^m\!<\!\iin$ for any~$m\ge 1$. Here we show that this polynomial bound on $T$ is optimal. 
\begin{proposition}\label{PP:aps}
Under Conditions {\rm(A)-(D)}, assume that for some non-cons\-tant function $f\in L_b(H)$  with $f(0)\neq 0$  and  initial  condition $u\in H$, there is a sequence   $r_k$ going to zero as $k\to\iin$ and  a random integer $T\ge1$ such that
\be\label{E:5.4}
\left|\f{1}{k}\sum_{n=1}^k f(u_n)-\lag f, \mu\rag\right|\le r_k \q\text{ for } k\ge T.
\ee
Then $T$ has an infinite exponential moment, i.e., 
$$
\e_u e^{\al T}=+\iin\q \q \text{ for any }\al>0.
$$
\end{proposition}
\bp
{\it Step 1: Contradiction argument.}
Suppose that for some $\al\!>\!0$ and~$M\!>\!0$,  
\be\label{E:5.5}
\e_u e^{\al T}\le M.
\ee
Let $\De>0$ be the number in Theorem \ref{T:1.2}. Up to multiplying   $f$ by a small positive constant and deducing another one, we may assume that $\Osc(f)\le \De$, $\| f\|_\ty\le \al$, and $\lag f, \mu \rag=0$.
Then, by \eqref{E:1.6},    
\be\label{E:5.6}
e^{-Q(f)k}\e_u \exp\left( \sum_{n=1}^kf(u_n)\right)\to h_f(u) \q\q\text{ as }k\to\iin,
\ee where $Q(f)=\log \lambda_f$. In Step 2, we will show that,   up to replacing $f$ by $-f$, we~have 
\begin{equation}\label{E:5.7}
Q(f)> 0.
\end{equation}
 We infer from \ef{E:5.4} that
$$
\sum_{n=1}^k f(u_n)\le \al \,T +\ch_{\{k\ge T\}}\sum_{n=1}^k f(u_n)\le \al\, T+k\,r_k,
$$
which, together with \ef{E:5.5}, implies
$$
\e_u \exp\left( \sum_{n=1}^kf(u_n)\right)\le e^{kr_k}\e_u e^{\al T}\le M e^{kr_k}.
$$
Combining this with \ef{E:5.6}, \ef{E:5.7}, and convergence $r_k\to 0$, we get $h_f(u)=0$, which is a contradiction.  

\smallskip
{\it Step 2: Proof of \eqref{E:5.7}.} As $Q: L_b(H)\to \rr $ is convex and $Q({\bf 0})=0$, up to replacing $f$ by $-f$, 
 we can assume\,\footnote{In fact $Q(f)\ge \langle f, \mu\rangle=0$ for any $f\in L_b(H)$. This follows from \ef{e57.3} and $I(\mu)=0$; see~(4.6) of \cite{JNPS-2012}.} that $Q(f)\ge0$.  
Let us suppose  that $Q(f)=0$.  Then from \eqref{E:1.6} we conclude that  
$$
\sup_{k\ge1}\E_\mu  \exp\left(\sum_{n=1}^k   f(u_n)\right)<\infty,
$$  where $\E_\mu$   is the expectation corresponding to the stationary measure. This~implies   that
$$  k^{-1}\,\E_\mu  \left(\sum_{n=0}^k   f(u_n)\right)^2\to 0  \quad \text{ as $k\to \ty.$}
$$
Combining this with   Proposition~\ref{PPP:pro}, we get    $f\equiv 0$.~This contradicts the assumption that~$f$ is non-constant and completes the proof of the proposition.
\ep

\begin{proposition}\label{PPP:pro}Under the conditions of Proposition~\ref{PP:aps}, the following  
  limit exists for any   $f\in L_b(H)$ with $\lag f, \mu\rag=0$: 
\begin{equation}\label{E:5.8}
  k^{-1}\,\E_\mu  \left(\sum_{n=0}^k   f(u_n)\right)^2\to \sigma_f^2  \quad \text{ as $k\to \ty.$}
\end{equation}	Moreover, if  $f(0)\neq 0$, then $\sigma_f\neq0$.
\end{proposition}
\begin{proof} 
 This result is a discrete-time version of Proposition~4.1.4 in~\cite{KS-book}, and the proof is essentially the same except that here we do not have irreducibility of the process.  	 By the Markov property and the stationarity of~$\mu$, we have 
  \begin{align*}
  	\e_\mu \left(\sum_{n=0}^k   f(u_n)\right)^2&=	\e_\mu \sum_{n=0}^k \sum_{r=0}^k f(u_r) f(u_n)\\&= 2	\sum_{r=0}^k \sum_{n=r}^k \e_\mu  \left(f(u_r)\e_\mu( f(u_n)|\FF_r)\right)-	\e_\mu \left( \sum_{n=0}^k   f^2(u_n)\right)\\&= 2	\sum_{r=0}^k \sum_{n=r}^k \e_\mu  \left(f(u_r)(\PPPP_{n-r} f)(u_r)\right)-	(k+1)\lag f^2, \mu \rag  \\&= 2	\sum_{r=0}^k \sum_{n=r}^k  \lag f  \PPPP_{n-r} f,\mu\rag -	   (k+1)\lag f^2, \mu \rag \\&= 2	\sum_{n=0}^k (k+1-n)  \lag f  \PPPP_n f,\mu\rag -	   (k+1)\lag f^2, \mu \rag.  
  \end{align*}Dividing this relation by $k$ and passing to the limit as $k\to \ty$, we get \eqref{E:5.8} with $\sigma_f^2=2\lag gf,\mu\rag-\lag f^2,\mu \rag$ and 
  $g(u)=\sum_{n=0}^\ty \PPPP_nf(u)$ for $u\in H$. Note that, by exponential mixing \eqref{E:0.2}, the function $g:H\to \rr$ is well defined and bounded on bounded sets of $H$.
  
  To prove the second part of the proposition, let us assume that $\sigma_f=0$ and
 consider Gordin's martingale approximation
  $$
  M_k:=\sum_{n=0}^\infty\left(\e_u(f(u_n)|\FF_k)-\e_u(f(u_n)|\FF_0)\right), \quad u\in H, \quad k\ge0,
  $$where    $\FF_k$  is the filtration corresponding to the Markov   process $(u_k, \pP_u)$. We shall use the following equality for these approximations~\footnote{Equality \eqref{3ERT} follows immediately from the Markov property:   
  \begin{align*}
  M_k&= \sum_{n=0}^{k-1}   f(u_n) + 	\sum_{n=k}^\infty \e_u(f(u_n)|\FF_k)-\sum_{n=0}^\infty\e_u(f(u_n)|\FF_0)\\&= \sum_{n=0}^{k-1}   f(u_n) + 	\sum_{n=k}^\infty (\PPPP_{n-k} f)(u_{k})-\sum_{n=0}^\infty(\PPPP_n f)(u) \\&=\sum_{n=0}^{k-1}   f(u_n) + 	g(u_{k})-g(u). 
  \end{align*}} 
\begin{equation}\label{3ERT}
 \sum_{n=0}^{k-1} f(u_n)=   M_k- g(u_k)+g(u).
\end{equation}

  Repeating the arguments of (4.19) in~\cite{KS-book}, we see that 
\begin{equation}\label{PSE:s}
  \pP_\mu\{M_k=0 \quad \text{for all $k\ge0$}\}=1.
 \end{equation}Let us show that this equality implies that $f\equiv0$. Indeed, assume that $f(0)>0$ (the case $f(0)<0$ is similar), and let $B$ be a ball in $H$ centred at zero     such that~$f(u)> \varepsilon$ for any $u\in B.$ Using  the facts that $S(0)=0$ (which follows~from Condition~{\rm(A)}),~$0\in \KK$ (see Condition~{\rm(D)}), and~$0\in \supp \mu $ (which  follows from~\eqref{E:1.2} and  Condition~{\rm(D)}), we see that 
 $$
 \pP_\mu\{u_n\in B: n=0,\ldots, k\}>0 \quad \text{for any $k\ge0$}. 
 $$
This implies that 
\begin{equation}\label{3ERT2}
 \pP_\mu\left\{\sum_{n=0}^{k-1} f(u_n)>k\varepsilon \right\}>0 \quad \text{for any $k\ge1$}. 
\end{equation}
  Let $C>0$ be such that $|g(u)|\le C$ for $u\in B$. Then $|g(u_k)-g(u)|\le 2 C$ if~$u_k, u\in B$. If $k\ge1$ is so large that $k\varepsilon>2C$, combining    \eqref{3ERT} and  \eqref{3ERT2}, we see that \eqref{PSE:s} cannot hold. This contradiction shows that $\sigma_f\neq 0$.
  \end{proof}

\subsection{The speed of attraction}
For any $\es>0$, let us introduce the random variable
$$
\nnn^\es(\om)=\#\left\{m\ge1  : u_m(\omega)\notin\aaa_\es \right\},
$$
where, as before, $\aaa_\es$ is the $\es$-neighborhood of $\aaa$ in $H$.
\begin{proposition}\label{P:5.1}
Let Conditions  {\rm(A)-(D)} be fulfilled.~Then, for any $\es>0$ and~$u\in H$, the random variable $\nnn^\es(\omega)$ is   $\pp_u$-almost surely finite. 
Moreover, there is a positive constant  $\al$ not depending on $\es$ such that  
\be\label{E:5.1}
\sup_{u\in B_R}\e_u e^{\al\nnn^\es}<\iin \q\q \text{ for any } R>0.
\ee
If in addition Condition {\rm(E)} is satisfied, then \ef{E:5.1} holds     with any $\al>0$.
\end{proposition}

 \bp
 Let $\De>0$ be the number entering Lemma \ref{L:3.3}. The proposition will be established if we show that there is a positive constant $\Lambda$ depending on $\De, \es,$ and~$R$ such that  \be\label{E:5.2}
\pp_u\{\nnn^\es\ge m\}\le \Lambda e^{-\De m} \q\q\text{ for any } u\in B_R \,\, \text{and } m\ge 1.
\ee Indeed, then inequality \ef{E:5.1} will hold with $\al=\De/2$.
To this end,  we consider a   function $V\in L_b(H)$ that vanishes on $\aaa$, equals~$\De$ outside $\aaa_\es$, and satisfies $0\le V\le \De$ on $H$.  
It follows that $\Osc(V)\le \De$.   Moreover, since $V$ vanishes on $\aaa$, we have~$\lm_V=1$ so that inequality \eqref{E:3.4}  
 holds true for this $V$.
    Let us introduce the random variable
$$
\nnn^\es_k(\om)=k\wedge \nnn^\es(\om)
$$and note that
$$
 \PPPP_k^V\mek(u)\ge\e_u  \left\{\ch_{\{\nnn^\es_k\ge m\}}\exp\left(\sum_{n=1}^k V(u_n)\right)\right\}\ge e^{\De m}\pp_u\{\nnn^\es_k\ge m\}.
$$
Letting $k$ go to infinity and using \eqref{E:3.4},  we arrive at \ef{E:5.2}.    Now if Condition~(E) is also fulfilled, by Lemma \ref{L:3.5}, the above $\De>0$ can be chosen arbitrarily large and thus so can be $\al$.
\ep

  \subsection{Kick-forced PDEs}\label{S:5.5}

Theorems \ref{T:1.1}-\ref{T:1.3} can be applied to a large class of dissipative PDEs perturbed by a {\it random kick force}.~In this section,  we  discuss the validity of Conditions~(A)-(E) for the Navier--Stokes, the complex Ginzburg--Landau, and the Burgers~equations.

 \subsubsection{2D Navier--Stokes system} 
  
 Let us consider the 2D Navier--Stokes (NS) system    for incompressible fluids:              \begin{equation} \label{E:5.14}
	\p_t u+\langle u,\nabla\rangle u-\nu\Delta u+\nabla p=\eta(t,x), \quad \diver u=0, \quad x\in D,
\end{equation}
where $D\subset \rr^2$ is a bounded domain     with smooth boundary~$\p D$,  $\nu>0$ is   the viscosity, $u=(u_1,u_2)$ and~$p$ are unknown   velocity field and pressure,~$\eta$ is an external random force, and   $\lag u, \nabla \rag=u_1\partial_1+u_2\partial_2$.  
  We denote by $H$ the  $L^2$-space of divergence free vector fields
$$
H= \left\{u\in L^2(D,\rr^2):\diver u=0\,\,\,\text{in $D$}, \lag u, {{\boldsymbol{\mathit n}}}\rag =0 \mbox{ on $\p D$}\right\}
$$ endowed with the norm $\|\cdot\|$, where  ${{\boldsymbol{\mathit n}}}$ is the outward unit normal   to~$\p D$.~By projecting~\eqref{E:5.14} to~$H$, we   eliminate the pressure and obtain an evolution system for  the velocity field  (e.g., see Section~6 in Chapter~1 of~\cite{lions1969}):
\begin{equation} \label{E:5.15}
\p_t u +B(u) + \nu L  u= \Pi\eta(t,x),  
\end{equation} 
where    $\Pi$ is the   orthogonal projection  onto~$H$ in~$L^2$ (i.e., the Leray projection), $L=-\Pi \Delta$ is the Stokes operator,  and  $B(u)=\Pi(\langle u,\nabla \rangle u)$.  
We assume that $\eta$   
 is a   random kick force  of the form
\begin{align}
\eta(t,x)=\sum_{k=1}^\ty \delta(t-k)\eta_k(x), \label{E:5.16}
 \end{align}
 where $\delta$ is the Dirac measure concentrated at zero and $\eta_k$ are i.i.d.~random variables   in $H$ satisfying Condition~{\rm(D)} with respect to an    orthonormal basis~$\{e_j\}$   formed by the eigenvectors of   $L$.~Under these assumptions, the trajectory~$u_t$ of~\eqref{E:5.15} is   normalised to be right-continuous and it is
   completely determined  by its restriction $u_k$ to integer times.  If we denote by~$S:H\to H$ the time-1 shift along the trajectories of~\eqref{E:5.15} with $\eta=0$, then 
 the sequence~$\{u_k\}$ satisfies \eqref{E:0.1}. The validity of Conditions (A)-(C)  for this system is checked in Section~3.2.4 in~\cite{KS-book}.    
 \begin{proposition}\label{P:5.7} There is a number~$ \nu_*>0$   such that for   $\nu\ge \nu_*$, Condition~{\rm(E)} is satisfied for the  NS system with the metric inherited  from $H$.
\end{proposition}
\begin{remark}
Let us note that in the case of large viscosity $\nu$, the ergodicity of the Markov process $(u_k, \pp_u)$ associated with \ef{E:5.14} has a quite simple proof; see Exercice 2.5.9 in \cite{KS-book}. It seems however, that this assumption does not lead to an easy   proof of the multiplicative ergodic theorem due to the presence of the potential $V$, which under condition \rm(E) can have an arbitrarily large oscillation.
\end{remark}
\begin{proof}[Proof of Proposition \ref{P:5.7}]  We split the proof
  into two steps.

\smallskip
{\it Step 1.}~Let us first show that the number $\rho$ in Condition (B) can be chosen the same for any $\nu\ge 1$. Indeed, using the inequality   
$$\|S(u_0)\|\le  e^{-\alpha_1 \nu} \|u_0\| \quad \text{ for $u_0\in H$}
$$ and the fact that   $\pP\{\|\eta_1\|\le C\}=1$, we~get
$$
\|u_1\| \le  e^{-\alpha_1 \nu}\|u_0\|+ C 
$$ with probability $1$, where  $\alpha_1$ is the first eigenvalue of $L$ and $C>0$ does  not depend on~$\nu$.~This   shows that the ball $B_\rho$ of radius $\rho\ge C(1-e^{-\alpha_1 \nu})^{-1}$ is invariant for \eqref{E:0.1}. Moreover, if we choose $\rho\ge  2C(1-e^{-\alpha_1 \nu})^{-1}$, then \eqref{E:1.3} is satisfied.  We take $\rho=2C(1-e^{-\alpha_1 })^{-1}$.   

\smallskip
{\it Step 2.} Let us show that
$$
\|S(u_0)-S(v_0)\| \le \|u_0-v_0\| \quad \text{for $u_0,v_0\in \aA(B_{\rho+1})$}
$$
 if   $\nu$ is sufficiently large.~We denote by $u$ and $v$ the   solutions of~\eqref{E:5.15} issued from $u_0$ and $v_0$, respectively. Then  $w=u-v$ satisfies  
\begin{equation}\label{E:5.17}
\dot w + B(w, u) + B(v, w) + \nu Lw = 0,
\end{equation}
where $B(w, u)=\Pi(\langle w,\nabla \rangle u)$. 
Taking the scalar product of \eqref{E:5.17}  with~$2w$ in~$H$, using the   equality $\lag B(v,w),w\rag=0$ and the   estimate
$$
|\lag B(w,u), w\rag| \le C_1 \|u\|_1 \|w\|_1^2,
$$
where $\|\cdot\|_1$ is the norm in the Sobolev space $H^1(D,\rr^2)$, we get
$$
\partial_t\|w\|^2+2(\nu -C_1\|u\|_{1})\|w\|_{1}^2\le0,$$ where $C_1$   does not depend on $\nu$.~By taking the     scalar product of \eqref{E:5.15} for $\eta=0$  with~$2u$ in~$H$, it is easy to see that  
$$
C_2:=\sup_{\nu\ge1,\, u_0\in \aA(B_{\rho+1})} \int_0^1\|u(s)\|_1\dd s <\ty.
$$ Using the previous two inequalities together with the Poincar\'e inequality  and the
  Gronwall lemma, we infer
\begin{align*}
\|w(1)\|^2&\le \|w(0)\|^2\exp\left(-2\alpha_1\nu +   C_1\alpha_1\int_0^1\|u(s)\|_1\dd s\right)\\&\le \|w(0)\|^2 \exp\left(-2\alpha_1\nu    +C_1 \alpha_1 C_2\right) \le \|w(0)\|^2
\end{align*}
 for $\nu\ge \nu_*:=C_1 C_3/2$. Note that we even have a contraction for sufficiently large $\nu$.
\end{proof}

   \subsubsection{Complex Ginzburg--Landau equation} 
The situation is similar for the   complex Ginzburg--Landau  (CGL)    equation:  
\begin{equation}
\p_t u-(\nu+i)\Delta u+ia|u|^2u=\eta(t,x), \quad x\in D, \quad u\big|_{\p D}=0,\label{E:CGL}
\end{equation}
where $\nu,a>0$ are some numbers, $D\subset\rr^3$ is a bounded domain with smooth boundary~$\p D$, $u=u(t,x)$ is a complex-valued function, and~$\eta$ is a kick force   of the form~\eqref{E:5.16}.  We consider this equation in the complex space $H=H_0^1(D)$ endowed with the norm $\|\cdot\|_1$. The random variables   $\eta_k$ are assumed to be i.i.d. in~$H$ and of  the form
$$ 
\eta_k(x)=\sum_{j=1}^\infty b_j(\xi_{jk}^1+i\xi_{jk}^2)e_j(x),
$$
where $\{e_j\}$ is an orthonormal basis in~$H$ formed by the eigenvectors of   the Dirichlet Laplacian,   $b_j>0$ for all $j\ge1$, and $\xi_{jk}^i$ are independent real-valued random variables whose laws possess the properties stated in Condition~(D).  By Proposition 1.7 in \cite{JNPS-2012}, Conditions~(A)-(C) hold for the CGL equation.
  The   following result is an analogue of Propositions~\ref{P:5.7}.    \begin{proposition}  
There is a number~$ \nu_*>0$   such that for   $\nu\ge \nu_*$, Condition~{\rm(E)} is satisfied for the CGL equation with the metric inherited  from $H$.
\end{proposition}
\begin{proof}[Sketch of the proof]   Inequalities (1.36) and (1.37)  in \cite{JNPS-2012}, combined with the arguments of Step 1 of the previous proof, show that the number $\rho$ in Condition~(B) can be chosen the same for any $\nu\ge 1$. We   check that for sufficiently large~$\nu$, 
$$
\|S(u_0)-S(v_0)\|_1 \le \|u_0-v_0\|_1 \quad \text{for $u_0,v_0\in \aA(B_{\rho+1}),$}
$$where $S:H\to H$ is the time-1 shift for  \eqref{E:CGL} with  $\eta=0$. Let $u$ and $v$ be the   solutions of~\eqref{E:CGL} issued from $u_0$ and $v_0$. Then  $w=u-v$ satisfies  
$$
\p_t w-(\nu+i)\Delta w+ia(|u|^2u-|v|^2v) = 0.
$$
 Multiplying this equation by~$2 \,\Delta w$ and integrating, we get 
\begin{align}
\p_t \|w\|_1^2&=2{\textup Re}\int_D\nabla\dot w\cdot\nabla\bar w\,\dd x
=-2{\textup Re}\int_D\dot w\,\Delta\bar w\,\dd x\notag\\
&=-2 {\textup Re}\int_D ((\nu+i)\Delta w-ia (|u|^2u-|v|^2v))\,\Delta\bar w\,\dd x \notag\\
&\le -2\nu\,\|\Delta w\|^2+2a\,\bigl\||u|^2u-|v|^2v\bigr\|\,\|\Delta w\|,
\label{E:5.19}
\end{align}where $\|\cdot\|$ is the $L^2$-norm.
The H\"older inequality and the   embedding $H_0^1\subset L^6$ imply that
$$
\bigl\||u|^2u-|v|^2v\bigr\|
\le C_1\bigl(\|u\|_1+\|v\|_1\bigr)^2\,\|w\|_1.
$$
Substituting this into~\eqref{E:5.19} and using the Poincar\'e inequality, we obtain
$$
\p_t\|w\|_1^2\le 
-\bigl(\nu\alpha_1-C_1(\|u\|_1+\|v\|_1)^4\bigr)\,\|w\|_1^2.
$$
 The Gronwall lemma and the standard inequality
 $$
C_2:=\sup_{\nu\ge1,\, t\in [0,1],\, u_0\in \aA(B_{\rho+1})}  \|u(t)\|_1\dd s <\ty 
$$ imply that
\begin{align*}
\|w(1)\|_1^2
&\le \exp\biggl(-\nu\alpha_1
+C_1\int_0^1(\|u_1\|_1+\|u_2\|_1)^4\dd s\biggr)\,\|w(0)\|_1^2\\
&\le \exp\bigl(-\nu\alpha_1+16\, C_1C_2^4 \bigr)\,\|w(0)\|_1^2\le \|w(0)\|_1^2
\end{align*}
 for $\nu\ge \nu_*:=16 \, C_1C_2^4/\alpha_1$.
\end{proof}

   \subsubsection{Burgers equation} 
   Let us consider the Burgers equation on the circle ${\mathbb S}=\rr/2\pi\Z:$
$$
\p_tu-\nu\p_x^2u+u\p_xu=\eta(t,x), 
$$
where   $\nu>0$  and $\eta$ is   of the form \eqref{E:5.16} with    i.i.d. random variables $\{\eta_k\}$   in
$$
H=\left\{u\in L^2({\mathbb S},\rr):\int_{\mathbb S} u(x)\dd x=0\right\}
$$ 
satisfying Condition (D) with an orthonormal basis $\{e_j\}$   formed by the eigenvectors of the  periodic Laplacian. The verification of Conditions (A)-(C) is similar to the case of the NS system.
\begin{proposition}       Condition  {\rm (E)}  is satisfied for the   Burgers equation   with any~$\nu>0$ with the metric inherited from $ L^1({\mathbb S},\rr)$. 
 \end{proposition}   
This proposition follows   immediately from      inequality 
 $$
\|S(u_0)-S(v_0)\|_{ L^1({\mathbb S},\rr)} \le \|u_0-v_0\|_{ L^1({\mathbb S},\rr)} \quad \text{for any  $u_0,v_0\in H$ and $\nu>0$}
$$ established in Section 3.3 of \cite{hormander1997}.

 \def\cprime{$'$} \def\cprime{$'$}
  \def\polhk#1{\setbox0=\hbox{#1}{\ooalign{\hidewidth
  \lower1.5ex\hbox{`}\hidewidth\crcr\unhbox0}}}
  \def\polhk#1{\setbox0=\hbox{#1}{\ooalign{\hidewidth
  \lower1.5ex\hbox{`}\hidewidth\crcr\unhbox0}}}
  \def\polhk#1{\setbox0=\hbox{#1}{\ooalign{\hidewidth
  \lower1.5ex\hbox{`}\hidewidth\crcr\unhbox0}}} \def\cprime{$'$}
  \def\polhk#1{\setbox0=\hbox{#1}{\ooalign{\hidewidth
  \lower1.5ex\hbox{`}\hidewidth\crcr\unhbox0}}} \def\cprime{$'$}
  \def\cprime{$'$} \def\cprime{$'$} \def\cprime{$'$}

\end{document}